\documentclass[12pt]{amsart}

\usepackage{amsmath, graphicx, cite, enumerate, comment}
\usepackage{amssymb, amsthm, amscd}
\usepackage{latexsym}
\usepackage{amssymb}
\usepackage[T1]{fontenc}	
\usepackage[english]{babel}
\usepackage{amsmath,amssymb}
\usepackage{amsfonts}
\usepackage{amscd}
\usepackage{cite}
\usepackage{marvosym}
\usepackage{mathrsfs}
\usepackage{amsthm}
\usepackage{hyperref}
\usepackage{verbatim}

\usepackage{tikz} \usetikzlibrary{arrows,shapes,patterns} 
\usepackage{lipsum} 

\newtheorem{theorem}{Theorem}[section]
\newtheorem{corollary}[theorem]{Corollary}
\newtheorem*{cor}{Corollary}
\newtheorem{lem}[theorem]{Lemma}
\newtheorem{proposition}[theorem]{Proposition}

\newtheorem*{problem}{Problem}

\theoremstyle{definition}
\newtheorem{definition}[theorem]{Definition}

\theoremstyle{remark}
\newtheorem{remark}[theorem]{Remark}
\newtheorem{rem}[theorem]{Remark}

\numberwithin{equation}{section}

\newcommand{\R}{\mathbb{R}}

\newcommand{\tu}{\tilde{u}}
\newcommand{\tZ}{\tilde{Z}}

\newcommand{\tg}{\overline{g}}

\newcommand{\tp}{\overline{\pi}}
\newcommand{\tR}{\overline{R}}

\newcommand{\op}{\overline{\pi}}

\newcommand{\tG}{\overline{\Gamma}}
\newcommand{\oG}{\overline{\Gamma}}
\newcommand{\G}{\Gamma}
\newcommand{\cg}{\check{g}}
\newcommand{\ck}{\check{k}}

\newcommand{\be}{\begin{equation}}
\newcommand{\ee}{\end{equation}}
\newcommand{\bd}{\begin{displaymath}}
\newcommand{\ed}{\end{displaymath}}

\begin{document}
\date{}
\title[Localizing solutions of the Einstein constraint equations]
{Localizing solutions of the Einstein constraint equations}
\author{Alessandro Carlotto}
\address{ETH - Institute for Theoretical Studies  \\
                 ETH \\
                 Z\"urich, Switzerland}
\email{alessandro.carlotto@eth-its.ethz.ch}
\author{Richard Schoen}
\address{Department of Mathematics \\
                 University of California\\
                 Irvine, CA 92697}
\email{rschoen@math.uci.edu}

\begin{abstract} We perform an \textsl{optimal localization} of asymptotically flat initial data sets and construct data that have positive ADM mass but are exactly trivial outside a cone of arbitrarily small aperture. The gluing scheme that we develop allows to produce a new class of $N$-body solutions for the Einstein equation, which patently exhibit the phenomenon of \textsl{gravitational shielding}: for any large $T$ we can engineer solutions where any two massive bodies do not interact at all for any time $t\in(0,T)$, in striking contrast with the Newtonian gravity scenario. 
\end{abstract}

\maketitle

\tableofcontents

\section{Introduction}

Scalar curvature plays a fundamental role in General Relativity. If $(M,g,k)$ is a space-like slice inside a spacetime $(\mathbb{L},\gamma)$ satisfying the Einstein field equations then, as a result of the Gauss equations relating the extrinsic and intrinsic geometry of such slice, the first and second fundamental forms of $M$ $(\textrm{respectively} \ g \ \textrm{and} \ k)$ satisfy the system
\begin{equation*}
\begin{cases}
\frac{1}{2}\left(R_{g}+\left(Tr_{g}k\right)^{2}-\left\|k\right\|^{2}_{g}\right)=\mu \\
Div_{g}\left(k-\left(Tr_{g}k\right)g\right)=J,
\end{cases}
\end{equation*}
for $\mu$ (the \textsl{mass density}) and $J$ (the \textsl{current density}) suitable components of the stress-energy tensor $T$, which describes the matter fields.
In the most basic of all cases, namely when $k=0$ (in which case we will say that the data $(M,g,k)$ are \textsl{time-symmetric}) the previous system, known as Einstein constraint equations, reduces to the single equation
\[ R_{g}=2\mu
\]
namely to a scalar curvature prescription problem.
As a result of a general physical axiom, the \textsl{dominant energy condition}, which postulates the energy density measured by any physical observer to be non-negative one requires the functional inequality $\mu\geq \left|J\right|_{g}$ to be satisfied and hence in the time-symmetric case considered above one obtains the restriction
\[ R_{g}\geq 0
\]
so that asymptotically flat spaces are always studied under the assumption that their scalar curvature be non-negative. Moreover, in the vacuum case (namely when no sources are present, $T=0$) the Einstein constraint equations reduce to the requirement that the scalar curvature vanishes at all points of the manifold in question. 

These local conditions have dramatic global consequences. The most remarkable of all of them is the \textsl{Positive Mass Theorem} (see \cite{SY79}), which essentially states that for asymptotically flat manifolds the ADM mass (a scalar invariant, introduced in the context of the Hamiltonian formulation of General Relativity and defined by a certain flux integral at spatial infinity) is indeed non-negative, and equals zero if and only if our manifold $(M,g)$ is globally isometric to the Euclidean space. For the purposes of this article, we should state the following important consequence\footnote{This relies on the main theorem in \cite{Wit81}, due to the fact that $\mathbb{R}^n$ is a spin manifold for any $n\geq3$.}.

\begin{cor}
Let $n\geq 3$ and $g$ be a Riemannian metric on $\mathbb{R}^{n}$ having non-negative scalar curvature. Suppose that there exists a compact set $K$ such that $g$ is exactly Euclidean outside of $K$. Then $g$ is the Euclidean metric on the whole $\mathbb{R}^{n}$.
\end{cor}

With a little bit more effort, one can get a sharper form of the previous statement (see Proposition \ref{pro:cap}) which implies that in fact an asymptotically flat metric cannot be localized inside sets that are asymptotically too small, for instance a cylinder or a slab between two parallel planes. Thus one is naturally led to the following basic question.

\begin{problem}
What is the optimal localization of an asymptotically flat metric of non-negative scalar curvature?
\end{problem}

For instance, can we construct an asymptotically flat metric of non-negative scalar curvature, positive ADM mass and which is exactly trivial in a half-space?

In this article, we give an essentially complete (and highly surprising) answer to such question, by developing a systematic method to localize a given scalar flat, asymptotically flat metric inside a cone of arbitrarily small aperture. In fact, we perform this construction for the general system of Einstein constraint equations, thus fully dealing with the coupled nonlinearities of the problem. We refer the readers to the next section for a precise statement of our gluing theorem (Theorem \ref{thm:main}), which requires some notation to be introduced. Here we will limit ourselves to a couple of important remarks.

First: as an immediate consequence of our gluing scheme we are able to produce data that are flat on a half-space and therefore contain plenty of stable (in fact: locally area-minimizing) minimal hypersurfaces, a conclusion which comes quite unexpected based on various recent scalar curvature rigidity results both in the closed and in the free-boundary case (see the works \cite{BBN10, Nun13, MM15, Amb15}). As a result, combining this fact with with the rigidity counterparts obtained by the first-named author, contained in \cite{Car13} and \cite{Car14},  we are able to provide a rather exhaustive answer to the fundamental problem of existence of stable minimal hypersurfaces in asymptotically flat manifolds (more generally: marginally outer trapped hypersurfaces in initial data sets). Furthermore, the reader shall notice that, again in the time-symmetric case, our solutions contain outlying volume-preserving stable constant mean curvature spheres that enclose arbitrarily large volumes. A well-known formula \cite{FST09} due to X.-Q. Fan, P. Miao, Y. Shi and L.-F. Tam equating the ADM mass to the isoperimetric mass computed in terms of deficit of large coordinate balls (together with an important remark of G. Huisken) ensures that none of those CMC surfaces is in fact isoperimetric, at least for large enough enclosed volumes.

Second: one can essentially iterate our construction and get a new class of $N$-body solutions to the Einstein constraint equation which exhibit, following a definition by P. Chru\'sciel, the phenomenon of \textsl{gravitational shielding} in the sense that one can prepare data that do not have any interaction for finite but arbitrarily long times, in striking contrast with the Newtonian gravity scenario. This is the content of Theorem \ref{thm:Nbody}. Concerning the evolution of these solutions, we point out that our data contain an arbitrarily large piece of the original data, so if there is a trapped region in the original, it will still exist in the localized solution. In such a case the evolution will be incomplete by the Penrose singularity theorem.

\

Apart from this introduction, the present article consists of five sections: in Section \ref{sec:gluthm} we introduce some notation, state our gluing theorem and discuss the regularity of the data we produce, in Section \ref{sec:prelim} we give an outline of the proof and present the variational structure of the linearized constraints in suitable doubly weighted Sobolev spaces. The most technical parts of the proof are contained respectively in Section \ref{sec:lin} for what concerns the linear theory (where we prove the coercivity of the functional $\mathcal{G}$ by means of some basic estimates of independent interest) and in Section \ref{sec:nonlin} for what concerns the Picard scheme by which we solve the Einstein equations. Lastly, the construction of $N$-body solutions is presented in Section \ref{sec:Nbody}.

\

\

\textsl{Acknowledgments}. The authors would like to express their sincere gratitude to the anonymous referess for carefully proofreading the article and for suggesting a number of changes which significantly contributed to the improvement of this final version. We are also indebted to Christos Mantoulidis for the figures that appear in the article. During the preparation of this work, the authors were partially supported by NSF grant DMS-1105323.

\section{The gluing theorem}\label{sec:gluthm}

\subsection{Initial data}\label{subs:ids}

Let $(M,\cg,\ck)$ be an asymptotically flat initial data set for the Einstein equation of type $(n,l+1,\alpha,\check{p})$, so that:
\begin{itemize}
\item{$(M^{n}, \cg)$ is a $\mathcal{C}^{l+1,\alpha}$ complete asymptotically flat manifold with $\cg_{ij}(x)=\delta_{ij}+O^{l,\alpha}(|x|^{-\check{p}})$}
\item{$\ck$ is a symmetric $(0,2)-$tensor with $\ck_{ij}(x)=O^{l-1,\alpha}(|x|^{-\check{p}-1})$}
\item{the \textsl{vacuum} Einstein constraint equations are satisfied, namely
\begin{equation}
\begin{cases}
R_{\cg}+(Tr_{\cg}\ck)^{2}-\left\|\ck\right\|^{2}_{\cg}=0 \\
Div_{\cg}(\ck-(Tr_{\cg}\ck)\cg)=0.
\end{cases}
\end{equation}
}
\end{itemize}
Here $\frac{n-2}{2}<\check{p}\leq n-2$ and we are adopting the standard definitions of weighted H\"older spaces given, for instance, in \cite{EHLS11}. 
As announced in the introduction, this article is devoted to performing an \textsl{optimal localization} of such data by gluing them to the trivial triple $(\mathbb{R}^{n},\delta, 0)$ outside of a cone of given aperture.

\subsection{The content at infinity of a Riemannian metric}

In order to describe this problem of \textsl{optimal localization} in some detail, we need to start by giving the following relevant definition.

\begin{definition}
Let $(M,g,k)$ be an asymptotically flat initial data set with one end of type $(n,l,\alpha,p)$ for $n, l\geq 3$, $p>(n-2)/2$ and $\alpha\in(0,1)$. For $g$ such an asymptotically flat Riemannian metric, we set 
\bd
U=\left\{p\in M \ | \ Ric_{g}(p)\neq 0\right\}
\ed
and define \textsl{content at infinity} for the metric in question to be the \textsl{asymptotic size of the set} $U$, namely
\bd
\Theta(g)=\liminf_{\sigma\to\infty} \sigma^{1-n}\mathscr{H}^{n-1}\left(U\cap \mathbb{S}^{n-1}\left(\sigma\right)\right).
\ed
Here $\mathbb{S}^{n-1}(\sigma)=\left\{|x|=\sigma\right\}$ for asymptotically flat coordinates $\left\{x\right\}$.
\end{definition}

The \textsl{Positive Mass Theorem} gives, among its various implications, strong restrictions on the content at infinity of an asymptotically flat metric of non-negative scalar curvature.

\begin{proposition}\label{pro:cap}
Let $(M,g,k)$ be a time-symmetric, asymptotically flat initial data set of type $(n,l,\alpha,n-2)$ for $n,l\geq 3$ and $\alpha\in(0,1)$. Suppose that $n<8$ or $(M,g)$ is a spin manifold. If $R_{g}\geq 0$, then \textsl{either} $g$ is flat \textsl{or} $\Theta(g)>0$.
\end{proposition}
\begin{proof}
It is well-known that the ADM energy (we will use the standard definition, see for instance \cite{Bar86} or \cite{EHLS11}) of an asymptotically flat initial data set $(M,g,k)$ can be equivalently expressed in terms of the Ricci curvature of $g$: namely there exists a positive dimensional constant $\varpi_{n}$ (whose specific value depends on the normalization used in the definition of $\mathcal{E}$) such that
\[\mathcal{E}=-\varpi_{n}\lim_{\sigma\to\infty}\sigma\int_{|x|=\sigma}Ric_g(\nu,\nu)\,d\mathscr{H}^{n-1}.
\]
As a result, for $\sigma$ large enough we can write the trivial estimate:
\[ \frac{\mathcal{E}}{2}\leq\varpi_{n}\left|\sigma\int_{|x|=\sigma}Ric_g(\nu,\nu)\,d\mathscr{H}^{n-1}\right|\leq C\sigma^{1-n}\mathscr{H}^{n-1}\left(U\cap \mathbb{S}^{n-1}\left(\sigma\right)\right).
\]
If $g$ were not flat, then by the \textsl{Positive Mass Theorem} (see \cite{SY79, Wit81}) it would have a positive ADM energy, that is $\mathcal{E}>0$. But then, by means of the previous inequality we would have
\[ \Theta(g)\geq \frac{\mathcal{E}}{2C}>0
\]
which completes the proof.
\end{proof}

Roughly speaking, this proposition states that if an asymptotically flat metric $g$ of non-negative scalar curvature is not trivial then the region where it is not (Ricci) flat must contain a cone of positive aperture. Incidentally, we recall here that for asymptotically flat metrics Ricci flatness is in fact equivalent to flatness, as one can prove either by means of the Bishop-Gromov comparison theorem or by using harmonic coordinates, as was done by Schoen-Yau in the proof of the rigidity statement of the \textsl{Positive Mass Theorem}. 
Our gluing scheme provides a sort of converse to the previous statement, by asserting that for any cone we can construct non-trivial data localized inside that cone, and nowhere else. In order to state our theorem, we need to describe our setting with more precision (which we will do in the next subsection).

\subsection{Regularized cones}\label{subs:regcon}

Given an angle $0<\theta<\pi$ and a point $a\in \mathbb R^n$ with $|a|>>1$
we denote by $C_\theta(a)$ the region of $M$ consisting of the compact
part together with the set of points $p$ in the exterior region such that $p-a$ makes an
angle less than $\theta$ with the vector $-a$. Here we are tacitly assuming that the manifold $M$ has only end, which we can do with no loss of generality as our construction is patently local to a given end: such assumption will always be implicit in the sequel of this article.

If we are given two angles $0<\theta_1<\theta_2<\pi$ we consider the region
between the cones. We want to regularize this domain near the vertex $a$,
so we consider the region 
\[ \Omega_1=B_{1/2}(a)\cup (C_{\theta_2}(a)\setminus \overline{C_{\theta_1}(a)}).
\]
We see that the boundary of this region is given by $\partial\Omega_1=S_1\cup S_2$
where $S_1$ and $S_2$ are hypersurfaces with corners on $\partial B_{1/2}(a)$
\[ S_1=(\partial C_{\theta_1}(a)\setminus B_{1/2}(a))\cup (\partial B_{1/2}(a)\cap
C_{\theta_1}(a))
\]
and
\[ S_2=(\partial C_{\theta_2}(a)\setminus B_{1/2}(a))\cup (\partial B_{1/2}(a)\setminus
C_{\theta_2}(a)).
\]
We approximate $\Omega_1$ by a smooth domain $\Omega$ whose boundary consists
of a pair of smooth disjoint hypersursurfaces $\Sigma_1$ and $\Sigma_2$ such that
for $i=1,2$
\[ \Sigma_i\setminus B_1(a)=C_{\theta_i}(a)\setminus B_1(a).
\]
It follows that $M\setminus (\Sigma_1\cup\Sigma_2)$ is a disjoint union of three regions
$\Omega_I$, $\Omega$, and $\Omega_O$ where we refer to $\Omega_I$ as the inner
region, $\Omega$ the transition region, and $\Omega_O$ the outer region.

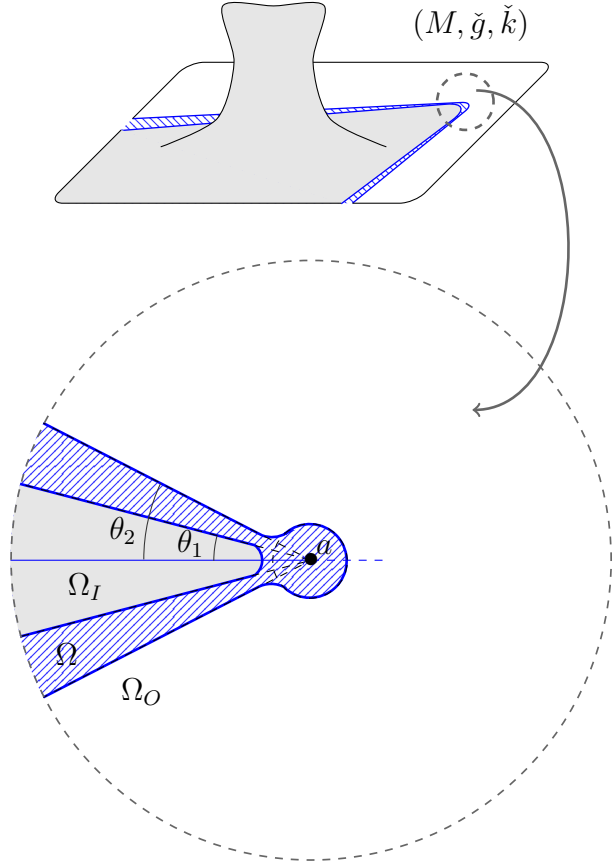
\begin{figure}[ht!]
	\begin{tikzpicture}
	
	\begin{scope}[scale=0.75, shift={(0, 8)}]
	\node at (5.4, 1) (MagSmall) {};
	\fill [pattern color=blue!80, pattern=north west lines, draw=blue, line width=0.5pt, rounded corners=2mm] plot coordinates {(-0.5, 0.5) (5.6, 0.8) (3.4, -1)};
	\fill [pattern color=blue!80, pattern=north west lines, draw=none, rounded corners=2mm] plot coordinates {(3.4, -1) (-2, -1) (-0.5, 0.5)};
	\fill [black!10, draw=blue, line width=0.5pt, rounded corners=4mm] plot coordinates {(-0.7, 0.29) (5.6, 0.8) (3.2, -1.01)}; 
	\fill [black!10, draw=black, rounded corners=2mm] plot coordinates {(3.2, -1) (-2, -1) (-0.7, 0.3)};         
	\draw [black, rounded corners=2mm] plot coordinates {(-0.5, 0.5) (0.5, 1.5) (7, 1.5) (4.5, -1) (3.4, -1)};          
	\fill [black!10, draw=black] plot [smooth] coordinates {(0, 0) (1, 0.5) (1.3, 1.5) (1.1, 2.5) (2, 2.5) (2.9, 2.5) (2.7, 1.5) (3, 0.5) (4, 0)};
	\draw [black!60, line width=1pt, dashed] (5.4, 0.8) circle (0.5);
	\node at (5.5, 2.2) {$(M, \check{g}, \check{k})$};              
	\end{scope}
	
	\begin{scope}[scale=0.5, shift={(4, 1)}]
	
	\clip (0, 0) circle (8);
	\node at (4, 4) (MagLarge) {};
	
	\begin{scope}     
	\filldraw [pattern color=blue!80, pattern=north east lines, draw=blue, line width=1pt] (-8, 4.1) -- (-1.3, 0.66) arc [radius=0.75, start angle=-110, delta angle=50] -- (-0.69, 0.71) arc [radius=0.98, start angle=132, delta angle=-264] -- (-0.69, -0.73) arc [radius=0.75, start angle=60, delta angle=60] -- (-8, -4.1);
	\filldraw [fill=black!10, draw=blue, line width=1pt] (-8.02, 2.105) -- (-1.5, 0.4) arc [radius=0.5, start angle=53, delta angle=-106] -- (-1.5, -0.4) -- (-8.02, -2.105); 
	\end{scope}
	
	\begin{scope}      
	\draw [black, dashed] plot coordinates {(-8, 4.1) (0, 0) (-8, -4.1)};
	\draw [black, dashed] plot coordinates {(-8, 2.1) (0, 0) (-8, -2.1)};
	\draw [black, dashed] (-0.03, -0.02) circle (0.99);                        
	\end{scope}
	
	\begin{scope}
	\draw [blue, ->] plot coordinates {(-8, 0) (-0.1, 0)};
	\draw [blue, dashed] plot coordinates {(0, 0) (2, 0)};
	\node at (0, 0) {$\bullet$};
	\node at (0.35, 0.35) {$a$};
	\end{scope}
	
	\begin{scope}
	\draw [black!80] (-4, 2.05) arc [radius=4.8, start angle=155, delta angle=25];
	\node at (-5, 0.8) {$\theta_2$};
	\draw [black!80] (-2.5, 0.656) arc [radius=2.59, start angle=165, delta angle=15];
	\node at (-3.2, 0.45) {$\theta_1$};
	\end{scope}
	
	\begin{scope}
	\node at (-6, -0.7) {$\Omega_I$};
	\node at (-6.5, -2.5) {$\Omega$};
	\node at (-4.5, -3.5) {$\Omega_O$};
	\end{scope}
	
	\draw [black!60, line width=1pt, dashed] (0, 0) circle (7.99);
	
	\end{scope}
	\path[every node/.style={font=\sffamily\small}, ->, black!60, line width=1pt] (MagSmall) edge [out=0, in=0] (MagLarge);
	
	\end{tikzpicture}

	\caption{Regularized cones and the gluing region $\Omega$.}
	\label{dgm:sector.figure}
\end{figure}

\subsection{Statement of the gluing theorem}

\begin{theorem}\label{thm:main} 
Assume that we are given a set of initial data $(M,\cg,\ck)$ as above together with angles $\theta_{1}, \theta_{2}$ satisfying $0<\theta_{1}<\theta_{2}<\pi$. Furthermore, suppose $\frac{n-2}{2}<p<\check{p}$. Then there exists $a_{\infty}$ so that for any $a\in\mathbb{R}^{n}$ such that $\left|a\right|\geq a_{\infty}$ we can find a metric $\hat{g}$ and a symmetric $(0,2)$-tensor $\hat{k}$ so that $(M,\hat{g},\hat{k})$ satisfies the vacuum Einstein constraint equations, $\hat{g}_{ij}=\delta_{ij}+O^{l-2,\alpha}(|x|^{-p})$ \ $\hat{k}_{ij}=O^{l-2,\alpha}(|x|^{-p-1})$ and 
\[
(\hat{g}, \hat{k})=\begin{cases}
(\cg, \ck) \ \ \textrm{in} \ \ \Omega_{I}(a) \\
(\delta, 0) \ \ \textrm{in} \ \ \Omega_{O}(a).
\end{cases}
\]
\end{theorem}

Let us now discuss the regularity of the data we produce.

\begin{rem}
From the viewpoint of the regularity of our data $(M,\hat{g}, \hat{k})$ we basically have two versions of Theorem \ref{thm:main}, which will be for brevity referred to as \textsl{finite regularity version} and \textsl{infinite regularity version}.

In the finite regularity version, we start with an (asymptotically flat) initial data set of type $(n,l+1,\alpha,\check{p})$ and get an initial data set of type $(n,l-1,\alpha,p)$. In other terms:
\begin{equation*}
\begin{cases}
\check{g}\in\mathcal{C}^{l,\alpha}_{loc} \\
\check{k}\in\mathcal{C}^{l-1,\alpha}_{loc}
\end{cases}
\  \Longrightarrow \
\begin{cases}
\hat{g}\in\mathcal{C}^{l-2,\alpha}_{loc}\\
\hat{k}\in\mathcal{C}^{l-2,\alpha}_{loc}.
\end{cases}
\end{equation*}
Therefore we face the well-known phenomenon of derivative loss that has been described both in \cite{Cor00} and \cite{CS06}. With more work it is possible to improve the theorem to remove this derivative loss and we will deal with this issue in a forthcoming paper. In the previous statement we shall assume that $l\geq 4$ so that in the most basic case we have
$(\check{g},\check{k})\in\mathcal{C}^{4,\alpha}_{loc}\times\mathcal{C}^{3,\alpha}_{loc}$ and $(\hat{g},\hat{k})\in\mathcal{C}^{2,\alpha}_{loc}\times\mathcal{C}^{2,\alpha}_{loc}$.

In the infinite regularity version, we start with an (asymptotically flat) initial data set where $M,\check{g},\check{k}$ are smooth and produces $(M,\hat{g},\hat{k})$ that are smooth as well. More precisely:
\begin{equation*}
\begin{cases}
\check{g}\in\mathcal{C}^{\infty}_{loc} \\
\check{k}\in\mathcal{C}^{\infty}_{loc}
\end{cases}
\  \Longrightarrow \
\begin{cases}
\hat{g}\in\mathcal{C}^{\infty}_{loc}\\
\hat{k}\in\mathcal{C}^{\infty}_{loc}.
\end{cases}
\end{equation*}
The proof of Theorem \ref{thm:main} is rather lenghty and technical, so we will only discuss it in the most basic case of finite regularity and for $l=4$. The modifications needed to obtain the general finite regularity version are straightforward and just of notational character. On the other hand, some changes are necessary for the infinity regularity version and, first of all, one needs to replace the angular weight of polynomial type with an angular weight of exponential type:
\[ \rho\simeq \phi^{2N} \ \ \leadsto \ \ \rho\simeq e^{-1/\phi}.
\] 
These aspects have been already dealt with in \cite{Cor00} and \cite{CS06} so we will not repeat them here.
\end{rem}

\begin{rem}
As the readers may have noticed, our whole discussion refers to the \textsl{vacuum} constraint equations. However, the method that we develop would work equally well in the case when the Einstein constraint equations have non-zero data on the right-hand side. The issue is that in the latter case, one should first of all extend the data $\check{\mu},\check{J}$ to data $\mu, J$ that equal $\check{\mu}, \check{J}$ in $\Omega_{I}$ and vanish in $\Omega_{O}$ and it is not entirely obvious that such an extension can always be made in a way that the dominant energy condition is satisfied (with respect to the metric $\hat{g}$). 
\end{rem}

\section{Some preliminaries}\label{sec:prelim}

Here and in the sequel we always assume $n\geq 3$ and we let $r\in \mathcal{C}^{\infty}(\R^{n})$ be any positive function which equals the usual Euclidean distance $|\cdot|$ outside of the unit ball (for a given set of coordinates). 

\subsection{Doubly weighted functional spaces}\label{subs:doubly}

Let $a\in\mathbb R^n$ with $|a|>>1$, and introduce coordinates $\left\{x\right\}$ centered at $a$. For angles
$0<\theta_1<\theta_2<\pi$ we consider the region $\Omega$ constructed in
the previous section (notice that the cones have vertex at $x=0$). 
We let $g$ be a Riemannian metric gotten by making a \textsl{rough patch} of $\cg$ and $\delta$ in the region $\Omega$: therefore $g$ agrees with $\cg$ in the connected component of the complement of $\Omega$ which contains the compact core of the manifold $M$ (which we called $\Omega_{I}$), and agrees with $\delta$ in the other connected component (which we called $\Omega_{O}$). We perform the same construction for $k$, which interpolates between $\ck$ and the trivial symmetric tensor. Of course, here we are making use of an \textsl{angular} cutoff function $\chi$ (namely a function only depending on the angle between a given point and the vector $-a$) with rapid decay at $\Sigma_{2}$ and such that $1-\chi$ rapidly decays at $\Sigma_{1}$. Finally, we observe that such $(g,k)$ satisfy the very same decay properties as $(\cg,\ck)$ even though they will not in general be a solution for the Einstein constraint equations. The main purpose of this work is to present a \textsl{deformation scheme} which allows to obtain such a goal. \newline
For $q>0$ we introduce the weighted $\mathcal{L}^2$ Sobolev space
$\mathcal{H}_{k,-q}(\Omega)$ of functions (or tensors) defined as the completion of the smooth
functions with bounded support (no condition on $\partial\Omega$) with 
respect to the norm $\|\cdot\|_{\mathcal{H}_{k,-q}}$ given by
\[ \|f\|^2_{\mathcal{H}_{k,-q}}=\sum_{i=0}^k\sum_{|\beta|=i}\int_\Omega|\partial^\beta f|^2r(x)^{-n+2(i+q)}
\ d\mathscr{L}^{n}(x)
\]
where we recall that $r(x)$ is a smooth positive function with $r(x)=|x|$ outside of the unit ball and $\mathscr{L}^{n}$ is the Lebesgue measure on $\mathbb{R}^n$. The reader shall notice that due to the decay properties of the Riemannian metrics we consider in this article, such Sobolev spaces could have been equivalently defined by means of covariant derivatives with respect to $g$ (denoted, in the sequel, by $D_g$) and by the $n$-dimensional Hausdorff measure associated to $g$ (denoted, in the sequel, by $\upsilon$). The space $\mathcal{H}_{k,-q}(\Omega)$
consists of those $f$ which roughly decay like $|x|^{-q}$ along the cone and for which a derivative of order 
$i\leq k$ decays like $|x|^{-q-i}$. Whenever no ambiguity is likely to arise, we will simply adopt the symbol $\left\|\cdot\right\|_{k,-q}$ rather than $\|\cdot\|_{\mathcal{H}_{k,-q}}$, the latter being kept only for situations when different classes of functional spaces are used at the same time.

For $0<q<n$ we use the $\mathcal{L}^2$ pairing to identify the dual space $\mathcal{H}^*_{0,-q}$ with
$\mathcal{H}_{0,q-n}$ since
\[ |\int_\Omega f_1f_2\ d\mathscr{L}^n|=|\int_\Omega (f_1r^{(-n/2+q)})(f_2r^{(n/2-q)})\ d\mathscr{L}^n|\leq \|f_1\|_{0,-q}\|f_2\|_{0,q-n}.
\]

Throughout this article, we need to work in \textsl{doubly} weighted Sobolev spaces, namely we also need to introduce an angular weight which is a certain (large) power of the angular distance of a point from $\partial\Omega$. This is necessary in order to prove that the gluing we perform is smooth (more generally: regular enough) up to the boundary of the gluing domain.

More precisely we take $\rho=\phi^{2N}$ where $N$ will be chosen to be a large integer and 
$\phi$ is a positive weight function which is equal to $\theta-\theta_1$ near 
$\Sigma_1\setminus B_1(0)$ and equal to $\theta_2-\theta$ near $\Sigma_2\setminus B_1(0)$.
Near $\Omega\cap B_1(0)$ we assume that $\phi$ vanishes with nonzero gradient along
$\partial \Omega$. We let $\phi_0$ denote the maximum value of $\phi$ and we assume that
each set $\Omega_t=\{\phi\geq t\}$ for $0\leq t\leq \phi_0$ is (the closure of) a smooth domain which is a cone outside $B_1(0)$. We define the doubly weighted space $\mathcal{H}_{k,-q,\rho}(\Omega)$ using the norm 
\[  \|f\|^2_{\mathcal{H}_{k,-q,\rho}}=\sum_{i=0}^k\sum_{|\beta|=i}\int_\Omega|\partial^\beta f|^2r(x)^{-n+2(i+q)}
\rho(x)\ d\mathscr{L}^n(x).
\]
As above, we will often use the simpler notation $\left\|\cdot\right\|_{k,-q,\rho}$ in lieu of $\|\cdot\|_{\mathcal{H}_{k,-q,\rho}}$.

Given $a\in \mathbb{R}^{n}$ we let $s(x)=\left|x+a\right|$ and we remark that for any assigned angle $\theta_1$ (and any $\theta_2\in(\theta_1,\pi)$) there exists $a_{\ast}=a_{\ast}(\theta_1)$ such that whenever $|a|>a_{\ast}$
\begin{equation}\label{eq:min}
\min_{x\in\Omega}s(x)\geq 
\begin{cases}
|a|\sin\theta_1 & \ \textrm{if} \ \  \theta_1\in(0,\frac{\pi}{2}) \\  
\frac{|a|}{2} & \ \textrm{if} \ \ \theta_1\in\left[\frac{\pi}{2},\pi\right)
\end{cases}
\end{equation}
and there exists a constant $C>0$ (not depending on any of the parameters $n, a$ but only on  $\theta_{1}, \theta_{2}$ and on the way the function $r$ is defined, namely on the value of its positive lower bound) such that 
\begin{equation}\label{eq:equiv}
C^{-1}\leq \frac{s(x)}{r(x)}\leq C\left|a\right|, \    \textrm{for all} \ x\in\Omega.
\end{equation}
In the sequel of this article, we will always tacitly assume to work with large enough $|a|$ (that is to say $|a|>a_{\ast}$) namely in a regime where thiose basic estimates are valid.

In this article, we let $C$ denote any constant depending only on $g,k,n,p,N,\theta_{1},\theta_{2}$ (or a subset thereof) while we indicate by $C^{s}_{(-q)}$ (resp. $C^{r}_{(-q)}$) any function or tensor which is bounded (in $\Omega$) by a constant as above times $s^{-q}(x)$ (resp. $r^{-q}(x)$). In Subsection \ref{subs:coercLie} it will be important to distinguish between those constants that depend on $N$ and those that do not and thus we will specify the functional dependence of the constants in all of the statements. 
 
We will need pointwise bounds on solutions to solve the nonlinear problem. For a point
$x\in\Omega$ we let $d(x)$ denote the distance from $x$ to $\partial\Omega$ and we
observe the bounds
\begin{equation}\label{eq:equiv2} C^{-1}r(x)\phi(x)\leq d(x)\leq Cr(x)\phi(x)
\end{equation}
for some positive constant $C$. 

We will use interior Schauder estimates to obtain the pointwise bounds. To do this we define 
global weighted H\"older norms. For real numbers $k,l,m$ 
with $k\geq 0$ an integer, we define $\|\cdot\|_{k,\alpha}^{(l,m)}$ by
\[ \|f\|_{k,\alpha}^{(l,m)}=\sum_{i=0}^k\sum_{|\beta|=i}\sup_{x\in\Omega}r(x)^{-l} \phi(x)^{-m}d(x)^i|\partial^\beta f(x)|
+\sum_{|\beta|=k}\sup_{x\in\Omega}r(x)^{-l} \phi(x)^{-m} d(x)^{k+\alpha}[\partial^\beta f]_{\alpha,B_{d(x)/2}(x)}
\]
where $[\cdot]_{\alpha,U}$ denotes the H\"older coefficient on $U$. The completion of the space of $\mathcal{C}^{\infty}_{c}(\mathbb{R}^{n})$ functions with respect to this norm will be denoted by $\mathcal{C}^{k,\alpha}_{l,m}(\Omega)$. Moreover, we consider the average value $\bar{f}(x)$ of $|f(x)|$ taken over the ball of radius $d(x)/2$.
 
\subsection{Einstein constraint operators}

Given data $(M,g,k)$, it is here convenient to recall the definition of the momentum tensor 
\[\pi^{ij}=k^{ij}-Tr_{g}\left(k\right)g^{ij}
\]
so that the vacuum Einstein constraint equations take the compact form
\[ \Phi(g,\pi)=0, 
\]
where
\[\Phi(g,\pi)=(\mathscr{H}(g,\pi), Div_{g}\pi), \ \mathscr{H}(g,\pi)=R_g+\frac{1}{n-1}\left(Tr_{g}\pi\right)^{2}-\left|\pi\right|^{2}_{g}.
\]
In order to avoid ambiguities, we remark that the second component of such system is a vector field. It is easily checked that for any $p\in \left(0, n-2\right)$ the linear map 
\[d\Phi_{(g,\pi)}: \mathcal{M}_{2, -p}\times \mathcal{S}_{1,-p-1}\to \mathcal{H}_{0, -p-2}\times \mathcal{X}_{0,-p-2}
\]
is continuous. Here and below we are using the symbols $\mathcal{M},\mathcal{S}, \mathcal{X}$ to stress that we are referring to a certain class of tensors, and specifically:
\begin{itemize}
\item{$\mathcal{M}_{2,-p}$ denotes the space of symmetric $(0,2)$ tensors in $\mathcal{H}_{2,-p}(\Omega)$;}
\item{$\mathcal{S}_{1,-p-1}$ denotes the space of symmetric $(2,0)$ tensors $\mathcal{H}_{1,-p-1}(\Omega)$;}
\item{$\mathcal{X}_{0,-p-2}$ denotes the space of vector fields in $\mathcal{H}_{0,-p-2}(\Omega)$}
\end{itemize}
with obvious extensions for different decay exponents or weights at the boundary of the gluing interface.

 Correspondingly, for the adjoint map, which is defined by means of the equation
\[\int_{\Omega}\left[d\Phi_{(g,\pi)}[h,\omega]\cdot_{g}(u,Z)\right]\,d\upsilon=\int_{\Omega}\left[(h,\omega)\cdot_{g}d\Phi^{\ast}_{(g,\pi)}[u,Z]\right]\,d\upsilon
\]
one has that $d\Phi^{\ast}\in \mathcal{C}^{0}\left(\mathcal{H}_{2, -n+p+2}\times \mathcal{X}_{1,-n+p+2}\to \mathcal{M}_{0, -n+p}\times \mathcal{S}_{0,-n+p+1}\right)$ by virtue of the $\mathcal{L}^{2}$-duality mentioned above.


Rather similar mapping properties are true in doubly weighted functional spaces, and in particular we shall make use of the fact that 
\[
d\Phi^{\ast}\in \mathcal{C}^{0}\left(\mathcal{H}_{2, -n+p+2,\rho}\times \mathcal{X}_{1,-n+p+2,\rho}\to \mathcal{M}_{0, -n+p,\rho}\times \mathcal{S}_{0,-n+p+1,\rho}\right).
\] Incidentally, such a statement can be effectively deduced by means of Lemma \ref{lem:coarea}, which allows to transform functional inequalities in singly weighted functional spaces into corresponding inequalities in doubly weighted ones.  

The expressions of $d\Phi_{(g,\pi)}[h,\omega]$ and $d\Phi^{\ast}_{(g,\pi)}[u,Z]$ have been computed (for instance) in \cite{FM73} and \cite{FM75} and for the purpose of this work we will recall them in symbolic form
\begin{equation*}
\begin{cases} d\Phi^{(1)}_{(g,\pi)}[h,\omega]&=-\Delta_g(Tr_g(h))+Div_g(Div_g(h))-g(h,Ric_g)+\pi\ast\pi\ast h+\pi\ast\omega\\
 d\Phi^{(2)}_{(g,\pi)}[h,\omega]&=Div_g(\omega)+\pi\ast D_gh
\end{cases}
\end{equation*}
and
\begin{equation*}
\begin{cases} {d\Phi^{\ast}}^{(1)}_{(g,\pi)}[u,Z]&=-\left(\Delta_g u\right)g+Hess_g(u)-uRic_g+\pi\ast\pi\ast u+D_g\left(\pi\ast Z\right)\\
{d\Phi^{\ast}}^{(2)}_{(g,\pi)}[u,Z]&=-\frac{1}{2}\mathscr{L}_{Z}g+\pi\ast u
\end{cases}
\end{equation*}
where $D_{g}$ denotes the Levi-Civita connection associated to $g$, all differential operators are considered with respect to the background metric $g$ and $\mathscr{L}$ stands for the Lie derivative. Let us remark that $L_g^{\ast}u=-\left(\Delta_g u\right)g+Hess_g(u)-uRic_g$ is the adjoint of the linearised scalar curvature operator.  We recall here that given tensors $A$ and $B$ the symbol $A\ast B$ denotes any finite linear combination (over $\R$) of contractions, via the background metric, of the product $A\otimes B$. Of course the type of the resulting tensor is clear from the context and when we want to stress it we will introduce indices, as in $\left(A\ast B\right)_{ij}^{k}$. 

\subsection{Variational framework}\label{subs:varframe}

Extending the approach of \cite{CS06}, it is convenient to introduce the functional $\mathcal{G}:\mathcal{H}_{2,-n+p+2,\rho}\times \mathcal{X}_{1,-n+p+2,\rho} \to \mathbb{R}$ defined by
\[\mathcal{G}(u,Z)=\int_{\Omega}\left\{\frac{1}{2}\left[\left|{d\Phi^{\ast}}^{(1)}_{\left(g,\pi\right)}[u,Z]\right|^{2}r^{n-2p}\rho+\left|{d \Phi^{\ast}}^{(2)}_{\left(g,\pi\right)}[u,Z]\right|^{2}r^{n-2p-2}\rho\right]-(f,V)\cdot_{g}(u,Z)\right\}\,d\upsilon
\]
where we are taking
\[ f\in \mathcal{H}_{0, -p-2,\rho^{-1}}, \ \ V\in \mathcal{X}_{0,-p-2,\rho^{-1}}.
\]
If we let 
\[ h=r^{n-2p}\rho \ ({d\Phi^{\ast}}^{(1)}_{(g,\pi)}[u,Z]) \ \ \ \omega=r^{n-2p-2}\rho \ ({d\Phi^{\ast}}^{(2)}_{(g,\pi)}[u,Z])
\]
then the Euler-Lagrange equation for the functional $\mathcal{G}$ takes the form
\[ d\Phi_{(g,\pi)}[h,\omega]-(f,V)=0.
\]
Moreover, we see at once that $h$ and $\omega$ decay according to what we claimed in the statement of Theorem \ref{thm:main} and namely
\[ \left|h\right|_{g}\lesssim \left|x\right|^{-p}, \ \ \left|\omega\right|_{g}\lesssim \left|x\right|^{-p-1}.
\]
As a result, Theorem \ref{thm:main} for $\check{p}=n-2$ is saying that we can get arbitrarily close to the decay of harmonic data for any dimension $n\geq 3$.

\subsection{Outline of the proof}

The general strategy we are about to follow is to obtain the data $(M,\hat{g},\hat{k})$ from the rough patch $(M,g,k)$ by solving the Einstein constraint equations iteratively, and more specifically by means of a Picard scheme. To that aim, we will need to solve a sequence of linearized problems of the form $d\Phi_{(g,\pi)}[h_{i},\omega_{i}]=(f_{i},V_{i})$ and show that indeed the tensors $g+h_{i}$ and $\pi+\omega_{i}$ converge, in suitable functional spaces, to a solution of the nonlinear constraint system.
 While this conceptual scheme is similar to that presented in \cite{Cor00} and \cite{CS06} (see also \cite{CD03}), we need to face here a number of serious technical obstacles.
In order to solve the linear problems we will show that the functional $\mathcal{G}$ is coercive, that fact following from certain subtle Poincar\'e type estimates (which we will call \textsl{Basic Estimates}, see Proposition \ref{pro:basic1} and Proposition \ref{pro:basicIIup}) of independent interest. Moreover, as mentioned before, the whole construction is performed in doubly weighted functional spaces, since we will have to keep control, at the same time, of both the decay at infinity and of the regularity at the boundary of the gluing region $\partial\Omega$.

\section{Solving the linear problem}\label{sec:lin}

The scope of this section is to show how to solve the linearized constraint equations by proving the existence of critical points of the functional $\mathcal{G}$. We will initially assume to work at the trivial data, and specifically at the flat metric $g=\delta$ and then  get the general case by a perturbation argument. Moreover, we claim that it is enough to prove coercivity estimates in singly weighted Sobolev spaces, and namely with only radial but no angular weights. This follows from the following coarea type lemma.

\begin{lem}\label{lem:coarea}
Let $\zeta\in \mathcal{C}^{\infty}_{c}(\mathbb{R}^{n},\mathbb{R})$, let $\rho:\Omega\to\mathbb{R}$ be defined in terms of $\phi$ as in Subsection \ref{subs:doubly} and let $\tilde{\rho}:(0,\phi_0)\to\mathbb{R}$ be the smooth, monotone increasing function characterized by the identity $\tilde{\rho}(\phi(x))=\rho(x)$ for all $x\in\Omega$.  Then
\[ \int_{\Omega}\zeta \rho\ d\upsilon=\int_0^{\phi_0}\tilde{\rho}'(t)\int_{\Omega_t}\zeta\ d\upsilon\, d\mathscr{L}^{1}
\]
where $\Omega_{t}=\left\{x\in\Omega \ : \ \phi(x)\geq t\right\}$.
\end{lem}

\begin{proof} We use the following level set formula for any smooth function $\zeta$
on $\bar{\Omega}$ with bounded support
\[ \int_{\Omega}\zeta \rho\ d\upsilon=-\int_0^{\phi_0}\frac{d}{dt}\int_{\Omega_t}\zeta\rho\ d\upsilon\,d\mathscr{L}^{1}+
\int_{\Omega_{\phi_0}}\zeta\tilde{\rho}(\phi_0)\ d\upsilon. 
\]
Now we have, because of the standard coarea formula and integration by parts
\[ -\int_0^{\phi_0}\frac{d}{dt}\int_{\Omega_t}\zeta\rho\ d\upsilon\ d\mathscr{L}^1=-\int_0^{\phi_0}\tilde{\rho}(t)
\frac{d}{dt}\int_{\Omega_t}\zeta\ d\upsilon\ d\mathscr{L}^1=-\tilde{\rho}(\phi_0)\int_{\Omega_{\phi_0}}\zeta\ d\upsilon
+\int_0^{\phi_0}\tilde{\rho}'(t)\int_{\Omega_t}\zeta\ d\upsilon\ d\mathscr{L}^1.
\]
Combining these we have shown
\[ \int_{\Omega}\zeta \rho\ d\upsilon=\int_0^{\phi_0}\tilde{\rho}'(t)\int_{\Omega_t}\zeta\ d\upsilon\ d\mathscr{L}^1
\] 
as we had claimed.
\end{proof}

Now, let us suppose we have proved that a certain bounded operator $T:\mathcal{H}_{k_{1},-r_{1}}\to\mathcal{H}_{k_{2},-r_{2}}$ satisfies a functional inequality of the form
\[ \left\|f\right\|_{k_{1},-r_{1}}\leq C\left\|Tf\right\|_{k_{2},-r_{2}}
\]
for some uniform constant $C>0$. Then, because of the previous lemma we can obtain that in fact
\[ \left\|f\right\|_{k_{1},-r_{1},\rho}\leq C \left\|Tf\right\|_{k_{2},-r_{2},\rho}
\]
by taking the previous singly weighted estimate for each $\Omega_{t}$ and integrating in $t$, thanks to the positivity of $\tilde{\rho}'$.
Of course, such simple argument is true when the domain and target of the operator in question, say $T$, are spaces of tensors of any type and not necessarily scalar functions as we considered here in order to simplify the discussion.

\subsection{Poincar\'e-type estimates in conical domains}

We first consider the Hamiltonian constraint and thus show that at the Euclidean metric (with the notations defined above) $\left\|u\right\|_{2,-n+p+2}\leq C\left\|L^{\ast}u\right\|_{0,-n+p}$. That essentially follows from the second of the following Poincar\'e inequalities, of independent interest. 

\begin{lem}\label{lem:intparts} For any real number $q$ with $0<q<(n-2)/2$, and $q\neq (n-4)/2$
for $n\geq 5$  we have
\[ \left\|u\right\|_{0,-q}\leq C \left\|\partial u\right\|_{0,-q-1}
\]
as well as
\[ \|u\|_{2,-q}\leq C\|Hess(u)\|_{0,-q-2}.
\]
\end{lem}
\begin{proof} We recall that we prove the inequality for the Euclidean metric, denoting by
$Hess(u)$ the Euclidean hessian. We consider the function $v=|x|^{2-n+2q}$
and observe that for $|x|\neq 0$ we have
\[ \Delta v=2q(2-n+2q)|x|^{-n+2q}
\]
where $\Delta$ is the Euclidean Laplace operator. Under our assumptions this
implies that $\Delta v=- \delta |x|^{-n+2q}$ for a positive constant $\delta=2q(n-2-2q)$. Assume that
$u$ has bounded support and integrate by parts
\[ \int_{\Omega\setminus B_1(0)}u^2\Delta v\ d\mathscr{L}^{n}=(n-2-2q)\int_{\partial B_1(0)\cap\Omega}u^2|x|^{1-n+2q}\ d\mathscr{H}^{n-1}-2\int_{\Omega\setminus B_1(0)} u\partial u \cdot \partial v\ d\mathscr{L}^{n}
\]
where we have used that fact that $\Omega$ is a cone outside of $B_1(0)$, so that the other
boundary terms there vanish. It follows from the sign of the boundary term and the Schwarz inequality that
\[ \delta \|u\|_{0,-q,\Omega\setminus B_1(0)}^2=\delta\int_{\Omega\setminus B_1(0)}u^2|x|^{-n+2q}\leq 2\|u\|_{0,-q,\Omega\setminus B_1(0)} \|\partial u\|_{0,-q-1,\Omega\setminus B_1(0)},
\]
and so we have
\[ \int_{\Omega\setminus B_1(0)}u^2r^{-n+2q}\ d\mathscr{L}^{n}\leq C\int_{\Omega\setminus B_1(0)}|\partial u|^2
r^{-n+2q+2}\ d\mathscr{L}^{n}.
\]

Let $\zeta$ be a smooth cutoff function with support in $B_2(0)$ and with $\zeta=1$ on $B_1(0)$.
A standard Poincar\'e inequality implies
\[ \int_{\Omega}(\zeta u)^2\ d\mathscr{L}^{n}\leq C\int_{\Omega}|\partial(\zeta u)|^2\ d\mathscr{L}^{n}. 
\]
In turn, this gives
\[ \int_{\Omega\cap B_1(0)}u^2r^{-n+2q}\ d\mathscr{L}^{n}\leq C\int_{\Omega\cap B_2(0)}|\partial u|^2
r^{-n+2q+2}\ d\mathscr{L}^{n}
\]
since $r$ is bounded above and below by positive constants on $B_2(0)$. Summing this
with the previous inequality we obtain
\[ \int_{\Omega}u^2r^{-n+2q}\ d\mathscr{L}^{n}\leq C\int_{\Omega}|\partial u|^2 r^{-n+2q+2}\ d\mathscr{L}^{n}.
\]

To obtain the desired bound we now apply the previous argument to each partial derivative
of $u$, say $w=\partial_j u$. We may use essentially the same argument with the
function $v=|x|^{4-n+2q}$ provided that $q\neq (n-4)/2$. The two cases when
$4-n+2q<0$ and when $4-n+2q>0$ work similarly with a sign reversal. In both cases
the boundary term can be thrown away and we end up showing
\[ \int_\Omega |\partial u|^2r^{-n+2q+2}\ d\mathscr{L}^{n}\leq C\int_\Omega |Hess(u)|^2r^{-n+2q+4}\ d\mathscr{L}^{n}
\]
which can be combined with our first functional inequality to complete the proof.
\end{proof}

From this we can deduce the coercivity estimate for the adjoint of the linearized scalar curvature operator, which corresponds to the coercivity of ${d\Phi^{\ast}}^{(1)}$ in the time-symmetric case, namely when $\pi=0$.

\begin{proposition} \label{pro:basic1}(Basic Estimate I) For any real number $p$ with $\frac{n-2}{2}<p<n-2$ and $p\neq n/2$ for $n\geq 5$ we have
\[ \|u\|_{2,-n+p+2}\leq C\|L^*(u)\|_{0,-n+p} \ \textrm{for all} \ u\in \mathcal{H}_{2,-n+p+2}(\Omega).
\]
 Hence, thanks to Lemma \ref{lem:coarea} we deduce that
\[ \|u\|_{2,-n+p+2,\rho}\leq C\|L^*(u)\|_{0,-n+p,\rho} \ \textrm{for all} \ u\in \mathcal{H}_{2,-n+p+2,\rho}(\Omega).
\]

\end{proposition}

\begin{proof}
Since $q=n-p-2$ satisfies $0<q<(n-2)/2$ and $q\neq (n-4)/2$ (by our assumption), we may apply the lemma to obtain
\[  \|u\|_{2,-n+p+2}\leq C\|Hess(u)\|_{0,-n+p}.
\]
Thus to complete the proof it suffices to show
\[ \|Hess(u)\|_{0,-n+p}\leq C\|L^*(u)\|_{0,-n+p}.
\]
Recalling that we are working at the Euclidean metric (and will then deduce a general coercivity result by perturbation) we can simply take the trace in the definition of the operator $L^{\ast}$ thereby obtaining
\[ Tr(L^*(u))=(1-n)\Delta u,\ \mbox{or}\ \Delta u=-1/(n-1)Tr(L^*(u)),
\]
frow which it follows that 
\[ Hess(u)=L^*(u)-1/(n-1)Tr(L^*(u))\delta
\]
and therefore
\[ \|Hess(u)\|_{0,-n+p}\leq C\|L^*(u)\|_{0,-n+p}
\]
which provides the Basic Estimate at the Euclidean metric. 
\end{proof}

\subsection{Coercivity of the Lie operator}\label{subs:coercLie}

We now consider the vector constraint equation, and prove coercivity of the differential ${d\Phi^{\ast}}^{(2)}$ by first analyzing the decoupled case when $\pi=0$.
Let $\mathcal D$ denote the Killing operator acting on vector fields
\[ \mathcal D(Y)(Z,W)=D_ZY\cdot W+D_WY\cdot Z.
\]
\begin{lem}\label{pro:basic2} Given a real number $0<q<\frac{n-2}{2}$, there exists a positive constant $C$ such that 
\[ \int_\Omega |Y|^2r^{-n+2q}\ d\mathscr{L}^n\leq C\int_\Omega|\mathcal D(Y)|^2r^{2-n+2q}\ d\mathscr{L}^n
\]
for any smooth vector field $Y$ in $\Omega$ which has bounded support (no condition on $\partial\Omega$).
\end{lem}

\begin{proof} We can use the standard arguments to get the bound on a compact set,
so it suffices to consider vector fields $Y$ which are defined on a truncated conical subregion of
$\Omega$ (say $\Omega\setminus B_{1}$), have bounded support, and vanishing on the inner boundary of $\Omega$. This
can be done by replacing $Y$ by $\zeta Y$ where $\zeta$ is a cutoff function which
is one outside a fixed ball and zero inside a fixed smaller ball. 

We first obtain the bound on the radial component of $Y$. We will work in orthonormal
bases for which $e_n=\partial_r$, so this component is denoted $Y_n$. Following the very same 
argument presented in the proof of Lemma \ref{lem:intparts} we obtain
\[ \int_\Omega(Y_n)^2r^{-n+2q}\ d\mathscr{L}^{n}=- \frac{1}{q}\int_\Omega(Y_n\partial_r Y_n)r^{1-n+2q}\ d\mathscr{L}^{n}.
\]
Since $D_{e_n}e_n=0$, we have 
$\partial_rY_n=\frac{1}{2}\mathcal D(Y)_{nn}$, so we can easily get
\[ \int_\Omega(Y_n)^2r^{-n+2q}\ d\mathscr{L}^{n}\leq C\int_\Omega|\mathcal D(Y)|^2 r^{2-n+2q}\ d\mathscr{L}^{n}.
\]
This gives the desired bound for $Y_n$.

Away from the axis of $\Omega$ (namely from the line generated by the vector $a$) we define orthonormal vector fields $e_{n-1},e_n$ where
$e_n=\partial_r$, the unit radial vector, $e_{n-1}=r^{-1}\partial_\theta$, the unit
vector field tangent to the spheres pointing away from the axis. We use the
notation $Y_i=Y\cdot e_i$ for $i=n-1,n$. To obtain the general bound, we have as above 
\[ \int_\Omega|Y|^2r^{-n+2q}\ d\mathscr{L}^{n}=-\frac{1}{q}\int_\Omega(Y\cdot\partial_rY)r^{1-n+2q}\ d\mathscr{L}^{n}.
\]
We observe that since $e_n=\partial_r$, we may write $Y=Z+Y_ne_n$
where $Z$ is orthogonal to the radial direction. We then have
\[ Y\cdot\partial_rY=\mathcal D(Y)(e_n,Y)-D_YY\cdot e_n,
\]
and
\[ D_YY\cdot e_n=D_ZZ\cdot e_n+Y_n D_{e_n}(Y_ne_n)\cdot e_n
+Y_n D_{e_n}Z\cdot e_n+D_Z(Y_ne_n)\cdot e_n.
\]
This expression simplifies to
\[ D_YY\cdot e_n=-r^{-1}|Z|^2+Y_ne_n(Y_n)+Z(Y_n)=-r^{-1}|Z|^2+Y(Y_n).
\]
It follows that
\[ \int_\Omega|Y|^2r^{-n+2q}\ d\mathscr{L}^{n}\leq \frac{1}{q}\int_\Omega|Y||\mathcal D(Y)|r^{1-n+2q}\ d\mathscr{L}^{n}
+\frac{1}{q}\int_\Omega Y(Y_n)r^{1-n+2q}\ d\mathscr{L}^{n}.
\]
We may integrate the second term by parts to obtain
\begin{align*}  \int_\Omega|Y|^2r^{-n+2q}\ d\mathscr{L}^{n} & \leq \frac{1}{q}\int_\Omega|Y||\mathcal D(Y)|r^{1-n+2q}\ d\mathscr{L}^{n}
-\frac{1}{q} \int_\Omega Div(r^{1-n+2q}Y) Y_n\ d\mathscr{L}^{n} \\
& +\frac{1}{q}\int_{\partial \Omega}|Y_{n-1}Y_n|r^{1-n+2q}\ d\mathscr{H}^{n-1}
\end{align*}
where we have used that fact that $e_{n-1}$ is the unit normal vector to
$\partial \Omega$. Since $Div(Y)$
is bounded by a fixed constant times $|\mathcal D(Y)|$ we obtain
\[ \int_\Omega|Y|^2r^{-n+2q}\ d\mathscr{L}^{n}\leq C\int_\Omega|Y||\mathcal D(Y)|r^{1-n+2q}\ d\mathscr{L}^{n}
+C\int_\Omega(Y_n)^2r^{-n+2q}\ d\mathscr{L}^{n}+C\int_{\partial \Omega}|Y_{n-1}Y_n|r^{1-n+2q}\ d\mathscr{H}^{n-1}.
\]
This clearly implies from our previous bound on $Y_n$
\begin{equation}\label{eq:central}
 \int_\Omega|Y|^2r^{-n+2q}\ d\mathscr{L}^{n}\leq C\int_\Omega|\mathcal D(Y)|^2r^{2-n+2q}\ d\mathscr{L}^{n}+
C\int_{\partial \Omega}|Y_{n-1}Y_n|r^{1-n+2q}\ d\mathscr{H}^{n-1}.
\end{equation}
as claimed. 

It remains to handle the boundary term. We have
\[\int_{\partial \Omega}|Y_{n-1}Y_n|r^{1-n+2q}\ d\mathscr{H}^{n-1}\leq \varepsilon/2\int_{\partial \Omega}|Y_{n-1}|^2r^{1-n+2q}\ d\mathscr{H}^{n-1}+(2\varepsilon)^{-1} \int_{\partial \Omega}|Y_n|^2r^{1-n+2q}\ d\mathscr{H}^{n-1}
\]
for any $\varepsilon>0$ to be chosen small enough. We recall that the region 
$\Omega$ is defined (outside of $B_1$) by $\theta_{1}\leq \theta\leq \theta_{2}$, and we let $\Sigma_\alpha$ denote the
hypersurface $\{\theta=\alpha\}$ for any $\alpha\in (0,\theta_2]$ so that $\partial \Omega=\Sigma_{\theta_1}\cup\Sigma_{\theta_{2}}$. In particular, the function $\theta$ should not be confused with $\phi$, which is instead (roughly speaking) the angular distance from $\partial\Omega$.  We also let $\Theta_\alpha$ denote the cone $\{0\leq\theta\leq\alpha\}$ so that
$\Omega=\Theta_{\theta_2}\setminus \Theta_{\theta_{1}}$.

For any smooth function $u$ with bounded
support in $\Omega$ and vanishing in a fixed neighborhood of the origin we have from the
coarea formula
\[ \int_{\Theta_{\theta_{2}}\setminus \Theta_{\theta_{1}}}u^2r^{-n+2q}\ d\mathscr{L}^{n}=\int_{\theta_{1}}^{\theta_2}\int_{\Sigma_t}
u^2r^{1-n+2q}\ d\mathscr{H}^{n-1}\ d\mathscr{L}^{1}
\]
since $|\nabla\theta|=r^{-1}$. Now, let us set
\[ I(t)=\int_{\Sigma_t}u^2r^{1-n+2q}\ d\mathscr{H}^{n-1}.
\]
For any $\alpha\in(\theta_{1},\theta_{2})$ we also have, by differentiating and applying the fundamental
theorem of calculus together with the coarea formula
\begin{align*} I(\theta_2)-I(\alpha)&=\int_\alpha^{\theta_2}\int_{\Sigma_t}(\partial_\theta(u^2)+(n-2)\cot(t)u^2)r^{1-n+2q}
\ d\mathscr{H}^{n-1}\ d\mathscr{L}^{1}\\
&=\int_{\Theta_{\theta_{2}}\setminus \Theta_{\alpha}}(re_{n-1}(u^2)+(n-2)\cot(\theta)u^2)r^{-n+2q}\ d\mathscr{L}^{n}
\end{align*}
as well as
\begin{align*} I(\alpha)-I(\theta_{1})&=\int_{\theta_{1}}^{\alpha}\int_{\Sigma_t}(\partial_\theta(u^2)+(n-2)\cot(t)u^2)r^{1-n+2q}
\ d\mathscr{H}^{n-1}\ d\mathscr{L}^{1}\\
&=\int_{\Theta_{\alpha}\setminus \Theta_{\theta_{1}}}(re_{n-1}(u^2)+(n-2)\cot(\theta)u^2)r^{-n+2q}\ d\mathscr{L}^{n}.
\end{align*}

In order to bound $I(\theta_2)$, we may use such formula and the intermediate value theorem to find $\alpha\in (\theta_{1},\theta_{2})$ so that
\[ I(\alpha)\leq 2/(\theta_2-\theta_1)\int_\Omega u^2r^{-n+2q}\ d\mathscr{L}^{n}.
\]
We then have from the formula above (note that $|\cot(\theta)|$ is bounded from above when
$\theta\in [\theta_1,\theta_2]$)
\[ \int_{\Sigma_{\theta_2}}u^2r^{1-n+2q}\ d\mathscr{H}^{n-1}=I(\theta_2)\leq C\int_\Omega u^2r^{-n+2q}\ d\mathscr{L}^{n}
+2\int_{\Theta_{\theta_{2}}\setminus{\Theta_{\alpha}}}u(e_{n-1}u)r^{1-n+2q}\ d\mathscr{L}^{n}
\]
and similarly for $I(\theta_{1})$ modulo sign changes, where appropriate.
So, in the end can write
\begin{align*}\int_{\partial \Omega}u^2r^{1-n+2q}\ d\mathscr{H}^{n-1} & \leq C\int_\Omega u^2r^{-n+2q}\ d\mathscr{L}^{n}
+2\int_{\Theta_{\theta_{2}}\setminus{\Theta_{\alpha}}}u(e_{n-1}u)r^{1-n+2q}\ d\mathscr{L}^{n} \\
& -2\int_{\Theta_{\alpha}\setminus{\Theta_{\theta_{1}}}}u(e_{n-1}u)r^{1-n+2q}\ d\mathscr{L}^{n}.
\end{align*}

Taking $u=Y_{n-1}$ we observe that
\[ e_{n-1}Y_{n-1}=D_{e_{n-1}}Y\cdot e_{n-1}+Y\cdot D_{e_{n-1}}e_{n-1}=\frac{1}{2}
\mathcal D(Y)(e_{n-1},e_{n-1})-r^{-1}Y_n.
\]
Therefore we have
\[ \int_{\partial \Omega}(Y_{n-1})^2r^{1-n+2q}\ d\mathscr{H}^{n-1}\leq C\int_\Omega((Y_{n-1})^2+(Y_n)^2)r^{-n+2q}\ d\mathscr{L}^{n}
+C\int_\Omega|\mathcal D(Y)|^2 r^{2-n+2q}\ d\mathscr{L}^{n}.
\]

Taking $u=Y_n$ we have
\[ e_{n-1}Y_n=D_{e_{n-1}}Y\cdot e_n+Y\cdot D_{e_{n-1}}e_n=
\mathcal D(Y)(e_{n-1},e_n)-D_{e_n}Y\cdot e_{n-1}+r^{-1}Y_{n-1}.
\]
Now we rewrite
\[ D_{e_n}Y\cdot e_{n-1}=e_n(Y_{n-1})-Y\cdot D_{e_n}e_{n-1}=e_n(Y_{n-1})
\]
and thus
\begin{align*} \int_{\partial \Omega}(Y_n)^2r^{1-n+2q}\ d\mathscr{H}^{n-1}&\leq C\int_\Omega((Y_n)^2+|Y_{n-1}Y_n|)r^{-n+2q}\ d\mathscr{L}^{n}-2\int_{\Theta_{\theta_{2}}\setminus \Theta_{\alpha}}Y_ne_n(Y_{n-1})r^{1-n+2q}\ d\mathscr{L}^{n}\\
&+2\int_{\Theta_{\alpha}\setminus \Theta_{\theta_{1}}}Y_ne_n(Y_{n-1})r^{1-n+2q}\ d\mathscr{L}^{n}+C\int_\Omega|\mathcal D(Y)|^2 r^{2-n+2q}\ d\mathscr{L}^{n}.
\end{align*}
We can write the term in the second and third integrals as
\[ Y_ne_n(Y_{n-1})=e_n(Y_nY_{n-1})-e_n(Y_n)Y_{n-1}=e_n(Y_nY_{n-1})-
\frac{1}{2}Y_{n-1}\mathcal D(Y)(e_n,e_n).
\]
After an integration by parts we arrive at the inequality
\[ \int_{\partial \Omega}(Y_n)^2r^{1-n+2q}\ d\mathscr{H}^{n-1}\leq C\int_\Omega((Y_n)^2+|Y_{n-1}Y_n|+r|Y_{n-1}|
|\mathcal D(Y)|)r^{-n+2q}\ d\mathscr{L}^{n}+C\int_\Omega|\mathcal D(Y)|^2 r^{2-n+2q}\ d\mathscr{L}^{n}.
\]

We can finally put things together and complete the proof by estimating the boundary term as follows
\begin{align*} \int_{\partial \Omega}|Y_{n-1}Y_n|r^{1-n+2q}\ d\mathscr{H}^{n-1}&\leq C\varepsilon\int_\Omega|Y|^2r^{-n+2q}\ d\mathscr{L}^{n}\\
&+C\varepsilon^{-1}\int_\Omega((Y_n)^2+|Y_{n-1}Y_n|+r|Y_{n-1}||\mathcal D(Y)|+r^2|\mathcal D(Y)|^2)r^{-n+2q}\ d\mathscr{L}^{n}
\end{align*}
Combining this inequality above with \eqref{eq:central}, and fixing $\varepsilon$ small enough we can absorb the first
term to obtain
\[ \int_\Omega|Y|^2r^{-n+2q}\ d\mathscr{L}^{n}\leq C\int_\Omega((Y_n)^2+|Y_{n-1}Y_n|+r|Y_{n-1}||\mathcal D(Y)|+r^2|\mathcal D(Y)|^2)r^{-n+2q}\ d\mathscr{L}^{n}.
\]
From our bound on $Y_n$ and easy estimates we now obtain the desired conclusion.
\end{proof}

\begin{proposition}(Basic Estimate II)\label{pro:basicIIup}
Given a real number $0<q<\frac{n-2}{2}$ there exists a positive constant $C$ such that 
\[ \int_\Omega |\nabla Z|^2 r^{2-n+2q}\rho \ d\mathscr{L}^{n}\leq C\int_\Omega|\mathcal D(Z)|^2r^{2-n+2q}\rho\ d\mathscr{L}^{n}
\]
for any smooth vector field $Z$
in $\Omega$ which has bounded support (no condition on $\partial\Omega$).
Hence, thanks to Proposition \ref{pro:basic2} and Lemma \ref{lem:coarea} we have for all $p\in (\frac{n-2}{2},n-2)$ 
\[
\left\|Z\right\|_{\mathcal{H}_{1,-n+p+2,\rho}}\leq C \left\|\mathcal{D}(Z)\right\|_{\mathcal{M}_{0,-n+p+1,\rho}} \ \textrm{for all} \ Z\in \mathcal{X}_{0,-n+p+1,\rho}.
\]
\end{proposition}

\begin{proof}
By virtue of the previous Lemma \ref{pro:basic2} (and its obvious weighted counterpart, gotten by applying Lemma \ref{lem:coarea}), it is enough for us to prove that 
\[ \int_\Omega |\nabla Z|^2 r^{2-n+2q}\rho \ d\mathscr{L}^{n}\leq C\int_\Omega|\mathcal D(Z)|^2r^{2-n+2q}\rho\ d\mathscr{L}^{n}+C\int_{\Omega}|Z|^{-n+2q}\rho\,d\mathscr{L}^{n}.
\]
Now, because of the very definition of Lie derivative 
\[\int_{\Omega} (Z^{2}_{i;j}+Z^{2}_{j;i})r^{2-n+2q}\rho\, d\mathscr{L}^{n}\leq C\int_\Omega|\mathcal D(Z)|^2r^{2-n+2q}\rho\, d\mathscr{L}^{n}-2\int_{\Omega}Z_{i;j}Z_{j;i}r^{2-n+2q}\rho\,d\mathscr{L}^{n}
\]
and thus we are reduced to proving, for indices $i\neq j$ an inequality of the form
\begin{align*} -2\int_{\Omega}Z_{i;j}Z_{j;i}r^{2-n+2q}\rho\,d\mathscr{L}^{n} & \leq \varepsilon \int_{\Omega} |\nabla Z|^{2}r^{2-n+2q}\rho\, d\mathscr{L}^{n}\\
& +\varepsilon^{-1}\left(\int_\Omega|\mathcal D(Z)|^2r^{2-n+2q}\rho \,d\mathscr{L}^{n}+\int_{\Omega}|Z|^{-n+2q}\rho\,d\mathscr{L}^{n}\right)
\end{align*}
for some $\varepsilon$ small enough that the Dirichlet term can be absorbed back in the left-hand side.
In partial analogy with the strategy that has been followed for proving Lemma \ref{pro:basic2}, we will apply some integration by parts (or, more precisely, we will use the divergence theorem) and, as will be clear from the sequel of our argument, the delicate point will be to control the boundary terms that may possibly arise in doing that. Since the outer normal to $\Omega$ is given (modulo sign) by $e_{n-1}$, we will limit ourselves to treat the case when $i\neq n-1$ and $j=n-1$. Otherwise the proof is strictly simpler and in fact does not require any delicate estimate.
For any $t\in (0,\phi_{0})$ let us recall that $\Omega_{t}=\left\{p\in\Omega : \phi(p,\partial\Omega)\geq t\right\}$. Applying the divergence theorem in $\Omega_{t}$ we get
\[ -2\int_{\Omega_{t}}Z_{i;j}Z_{j;i}r^{2-n+2q}\,d\mathscr{L}^{n}\leq 2\int_{\Omega_{t}}Z_{j}Z_{i;ji}r^{2-n+2q}\,d\mathscr{L}^{n} + C\int_{\Omega_{t}}|Z_{j}Z_{i;j}|r^{1-n+2q}\,d\mathscr{L}^{n}
\]
for some positive constant $C$ depending on $n$ and $q$ only. The standard rearrangement trick allows to treat the second summand on the right-hand side, so we only need to get an upper bound for
\[ 2\int_{\Omega_{t}}Z_{j}Z_{i;ji}r^{2-n+2q}\,d\mathscr{L}^{n}.
\]
Since the background metric is Euclidean, possibly by introducing other lower order terms (that can be treated as above) we can interchange the order of derivatives from $D_{e_{i}}D_{e_{j}}$ to $D_{e_{j}}D_{e_{i}}$ and hence we need to deal with
\[ 2\int_{\Omega_{t}}Z_{j}Z_{i;ij}r^{2-n+2q}\,d\mathscr{L}^{n}.
\]
Applying the divergence theorem again, we obtain
\begin{align*} 2\int_{\Omega_{t}}Z_{j}Z_{i;ij}r^{2-n+2q}\,d\mathscr{L}^{n}  \leq & -2\int_{\Omega_{t}}Z_{i;i}Z_{j;j}r^{2-n+2q}\,d\mathscr{L}^{n}+2\int_{\partial\Omega_{t}}|Z_{j}Z_{i;i}|r^{2-n+2q}\,d\mathscr{H}^{n-1} \\
& + C\int_{\Omega_{t}}|Z_{j}Z_{i;i}|r^{1-n+2q}\,d\mathscr{L}^{n}. 
\end{align*}
Putting together the previous inequalities and applying the coarea type formula given by Lemma \ref{lem:coarea} we come to an inequality of the form
\[ \left\|Z\right\|^{2}_{\mathcal{H}_{1,-q,\rho}}\leq C\left[\left\|\mathcal{D}Z\right\|^{2}_{\mathcal{H}_{0,-q-1,\rho}}+\left\|Z\right\|^{2}_{\mathcal{H}_{0,-q,\rho}}+2N\sum_{i,j}\int_{0}^{\phi_{0}}\phi^{2N-1}\int_{\partial\Omega_{\phi}}|Z_{j}Z_{i;i}|r^{2-n+2q}\,d\mathscr{H}^{n-1}d\mathscr{L}^{1}\right]
\]
and thus 
\[\left\|Z\right\|^{2}_{\mathcal{H}_{1,-q,\rho}}\leq C\left[\left\|\mathcal{D}Z\right\|^{2}_{\mathcal{H}_{0,-q-1,\rho}}+\left\|Z\right\|^{2}_{\mathcal{H}_{0,-q,\rho}}+2N\sum_{i,j}\int_{\Omega}|Z_{j}Z_{i;i}|\phi^{-1}r^{1-n+2q}\rho\,d\mathscr{L}^{n}\right]
\]
where in both cases $C$ denotes a positive constant only depending on $n$ and $q$.
Therefore, since the last summand on the right-hand side can be bounded from above by
\[ 2N\int_{\Omega}|Z_{j}Z_{i;i}|\phi^{-1}r^{1-n+2q}\rho\,d\mathscr{H}^{n-1}d\mathscr{L}^{1}\leq N^{2}\int_{\Omega}Z^{2}_{i;i}r^{2-n+2q}\rho\,d\mathscr{L}^{n}+\int_{\Omega}Z^{2}_{j}\phi^{-2}r^{-n+2q}\rho\,d\mathscr{L}^{n}
\]
we have proven that in fact
\[ \left\|Z\right\|^{2}_{\mathcal{H}_{1,-q,\rho}}\leq C\left[N^{2}\left\|\mathcal{D}Z\right\|^{2}_{\mathcal{H}_{0,-q-1,\rho}}+\left\|Z\right\|^{2}_{\mathcal{H}_{0,-q,\rho}}+\int_{\Omega}|Z|^{2}\phi^{-2}r^{-n+2q}\rho\,d\mathscr{L}^{n}\right]
\]
which implies our conclusion once we show that
\[ \int_{\Omega}|Z|^{2}\phi^{-2}r^{-n+2q}\rho\,d\mathscr{L}^{n}\leq C\int_{\Omega}|Z|^{2}r^{-n+2q}\rho\,d\mathscr{L}^{n}+\frac{C}{N^{2}}\int_{\Omega}|\nabla Z|^{2}r^{2-n+2q}\rho\,d\mathscr{L}^{n}.
\]
To that aim, let us introduce an angular cut-off function $\xi$ that equals 1 for $\phi(\cdot,\partial\Omega)\leq \phi_{0}/3$ and 0 for $\phi(\cdot,\partial\Omega)\geq \phi_{0}/2$. It is obvious that 
\[ \int_{\Omega}(1-\xi)|Z|^{2}\phi^{-2}r^{-n+2q}\rho\,d\mathscr{L}^{n}\leq C\int_{\Omega}|Z|^{2}r^{-n+2q}\rho\,d\mathscr{L}^{n}
\]
for some constant $C$ only depending on $n, q, \theta_{1}, \theta_{2}$ and therefore we only need to produce an estimate for $\int_{\Omega}\xi|Z|^{2}\phi^{-2}r^{-n+2q}\rho\,d\mathscr{L}^{n}$.
If we apply the divergence theorem in $\Omega$ to the vector field $\xi|Z|^{2}r^{2-n+2q}\nabla\rho$ and exploit the fact that (on the support of $\xi$) $\Delta\rho\geq C N^{2}r^{-2}\phi^{-2}\rho$ (for some constant $C$ which does not depend on $N$) we get
\[\int_{\Omega}\xi|Z|^{2}\phi^{-2}r^{-n+2q}\rho\,d\mathscr{L}^{n}\leq \frac{C}{N}\left[\int_{\Omega}|\xi'||Z|^{2}\phi^{-1}r^{1-n+2q}\rho\,d\mathscr{L}^{n}+\int_{\Omega}\xi|Z||\nabla Z|\phi^{-1}r^{1-n+2q}\rho\,d\mathscr{L}^{n}\right]
\]
because $\nabla r \perp \nabla \phi$ and $\int_{\partial\Omega}\xi|Z|^{2}r^{2-n+2q}\nabla\rho\cdot\nu\,d\mathscr{H}^{n-1}\leq 0$. Thus, let us notice that (by the Cauchy-Schwarz inequality)
\[\int_{\Omega}\xi|Z||\nabla Z|\phi^{-1}r^{1-n+2q}\rho\,d\mathscr{L}^{n}\leq \left(\int_{\Omega}\xi|\nabla Z|^{2}r^{2-n+2q}\rho\,d\mathscr{L}^{n}\right)^{1/2}\left(\int_{\Omega}\xi|Z|^{2}\phi^{-2}r^{-n+2q}\rho\,d\mathscr{L}^{n}\right)^{1/2}
\]
and therefore, if we set
\[ F_{W}=\int_{\Omega}\xi|Z|^{2}\phi^{-2}r^{-n+2q}\rho\,d\mathscr{L}^{n}, \ F_{L}=\int_{\Omega}|\xi'||Z|^{2}\phi^{-1}r^{1-n+2q}\rho\,d\mathscr{L}^{n}, \ F_{D}=\int_{\Omega}\xi|\nabla Z|^{2}r^{2-n+2q}\rho\,d\mathscr{L}^{n}
\] 
we have just proven an inequality of the form
\[ F_{W}\leq \frac{C}{N} F_{L}+\frac{C}{N}F^{1/2}_{W}F^{1/2}_{D}.
\]
Let us now remark that since $\xi'=0$ near the boundary $\partial\Omega$ and, by scaling arguments, $|\xi'|\leq Cr^{-1}$ on the whole $\Omega$ we can write
\[ F_{L}\leq C\int_{\Omega}|Z|^{2}r^{-n+2q}\rho\,d\mathscr{L}^{n}.
\]
In order to conclude our proof, we distinguish two cases. If $F^{1/2}_{W}\leq \frac{C}{N}F^{1/2}_{D}+F^{1/2}_{L}$ then we are done as soon as we pick $N$ large enough to absorbe the Dirichlet integral in the left-hand side of our main inequality. If instead this is not the case, and thus $F^{1/2}_{W}-\frac{C}{N}F^{1/2}_{D}\geq F_{L}^{1/2}$ we obtain (from the inequality relating $F_{W}, F_{L}, F_{D}$) that
\[ F_{W}^{1/2}F_{L}^{1/2}\leq \frac{C}{N}F_{L}
\]
which is the same as 
\[ F_{W}\leq \frac{C}{N^{2}}F_{L}
\] and this completes the proof.
\end{proof}

\subsection{Existence of critical points}

In this subsection, we capitalize the effort spent in proving the coercivity inequalities for the adjoint constraint operators by deriving the existence of critical points for the functional $\mathcal{G}$ associated to the linearized problem.

First of all, let us observe that thanks to Lemma \ref{lem:coarea} we can turn Proposition \ref{pro:basic1} into a corresponding statement in doubly weighted Sobolev spaces, namely when the angular weight $\rho$ is also taken into account. Thus, we obtain the natural counterpart of Proposition \ref{pro:basicIIup}, which instead concerns the second component of the constraints. 




The following proposition is the key to solve the linear Einstein constraint system in our setting.

\begin{proposition}\label{pro:basic}
Let $n\geq 3$ and for any set of data $(M,\cg,\ck)$ as in the statement of Theorem \ref{thm:main} let $(M,g,k)$ be the triple defined in Subsection \ref{subs:doubly}. Fix a real number $\frac{n-2}{2}<p<\check{p}$ with $p\neq n/2$ if $n\geq 5$  . There exist constants $a_{\infty,L}$ and $C$ (depending only on $g, k, \theta_{1}, \theta_{2}, p$) such that uniformly for $|a|>a_{\infty}$
\[ \left\|\left(u,Z\right)\right\|_{\mathcal{H}_{2,-n+p+2,\rho}\times \mathcal{X}_{1,-n+p+2,\rho}}\leq C \left\|d\Phi^{\ast}_{\left(g,\pi\right)}\left[u,Z\right]\right\|_{\mathcal{M}_{0,-n+p,\rho}\times \mathcal{S}_{0,-n+p+1,\rho}}
\]
for all $u\in \mathcal{H}_{2,-n+p+2,\rho}$ and $Z\in \mathcal{X}_{1,-n+p+2,\rho}$.
\end{proposition}

\begin{proof}
The Basic Estimate I (Proposition \ref{pro:basic1}) and II (Proposition \ref{pro:basicIIup}) ensure the existence of a constant $C>0$ (that can be chosen uniformly for all suffciently large $|a|$) such that
	\[
	\left\|\left(u,Z\right)\right\|_{\mathcal{H}_{2,-n+p+2,\rho}\times \mathcal{X}_{1,-n+p+2,\rho}}\leq  C \left\|d\Phi^{\ast}_{\left(\delta,0\right)}\left[u,Z\right]\right\|_{\mathcal{M}_{0,-n+p,\rho}\times \mathcal{S}_{0,-n+p+1,\rho}}.
	\]
	
On the other hand, given the decay assumptions on $(\hat{g},\hat{k})$ (hence on $(g,k)$) it follows from a standard perturbation argument that for any $\varepsilon>0$ we can find $a_{\infty,L}=a_{\infty,L}(\varepsilon)$ such that
\[
|a|>a_{\infty,L} \ \Rightarrow \ 
\left\|d\Phi^{\ast}_{(g,\pi)}[u,Z]-d\Phi^{\ast}_{\left(\delta,0\right)}\left[u,Z\right]\right\|_{\mathcal{M}_{0,-n+p}\times \mathcal{S}_{0,-n+p+1}}\leq\varepsilon \left\|\left(u,Z\right)\right\|_{\mathcal{H}_{2,-n+p+2}\times \mathcal{X}_{1,-n+p+2}}
\]
and hence, by virtue of Lemma \ref{lem:coarea}
\[
\left\|d\Phi^{\ast}_{\left(g,\pi\right)}\left[u,Z\right]-d\Phi^{\ast}_{\left(\delta,0\right)}\left[u,Z\right]\right\|_{\mathcal{M}_{0,-n+p,\rho}\times \mathcal{S}_{0,-n+p+1,\rho}}\leq\varepsilon \left\|\left(u,Z\right)\right\|_{\mathcal{H}_{2,-n+p+2,\rho}\times \mathcal{X}_{1,-n+p+2,\rho}}.
\]	Thus, picking $\varepsilon_0=C/3$ we have that for all $|a|>a_{\infty,L}(\varepsilon_0)$ both previous inequalities are true, and therefore the triangle inequality ensures that
\[
\left\|\left(u,Z\right)\right\|_{\mathcal{H}_{2,-n+p+2,\rho}\times \mathcal{X}_{1,-n+p+2,\rho}}\leq \frac{2C}{3} \left\|d\Phi^{\ast}_{\left(g,\pi\right)}\left[u,Z\right]\right\|_{\mathcal{M}_{0,-n+p,\rho}\times \mathcal{S}_{0,-n+p+1,\rho}}
\]
which is what we wanted.
\end{proof}	

At this stage, we can use a direct method to find a (unique) global minimum for $\mathcal{G}$.

\begin{proposition}\label{pro:existence}
Let $n\geq 3$, let $\frac{n-2}{2}<p<\check{p}$ with $p\neq n/2$ if $n\geq 5$ and assume that the vertex $a$ satisfies the inequality $|a|>a_{\infty,L}$. For any $(f, V)\in \mathcal{H}_{0, -p-2,\rho^{-1}}\times \mathcal{X}_{0,-p-2,\rho^{-1}}$ there exists a unique $(\tu,\tZ)\in \mathcal{H}_{2,-n+p+2,\rho}\times \mathcal{X}_{1, -n+p+2,\rho}$ which minimizes the functional $\mathcal{G}$ on the Hilbert space $\mathcal{H}_{2,-n+p+2,\rho}\times \mathcal{X}_{1, -n+p+2,\rho}$.
\end{proposition}

\begin{rem}
As the reader may have noticed, we did not mention, in the statement of Theorem \ref{thm:main}, the exceptional value $p^{\ast}=p^{\ast}(n)=n/2$ related to the restrictions of the basic estimates for dimension $n\geq 5$. Indeed, we claim that this is not an issue when performing the gluing. The reason is very simple: given data $(M,\cg,\ck)$ with decay $\check{p}$ (in the usual sense) if one wanted to prove Theorem \ref{thm:main} for $p=p^{\ast}$ then by simply performing our whole construction with weight $\frac{p^{\ast}+\check{p}}{2}$ (which is larger than $p^{\ast}$, hence certainly not exceptional) one would produce a triple $(M,\hat{g},\hat{k})$ which does solve the Einstein constraint equation and whose decay is actually \textsl{faster} than the initial requirement (thereby satisying the conclusion of Theorem \ref{thm:main} for $p=p^{\ast}$). In other words (and more generally) for fixed $\check{p}$ the validity of Theorem \ref{thm:main} for some $p''<\check{p}$ implies the validity of the same assertion for all $p'\in(\frac{n-2}{2},p'')$.
\end{rem}

\begin{proof} The argument follows the direct method of the Calculus of Variations. Indeed, the functional $\mathcal{G}$ is bounded from below on $\mathcal{H}_{2,-n+p+2,\rho}\times \mathcal{H}_{1, -n+q+1,\rho}$ for its very definition implies
\[
\begin{split}
\mathcal{G}(u, Z)\geq C_1\left\|d\Phi^{\ast}_{\left(g,\pi\right)}[u,Z]\right\|^{2}_{\mathcal{M}_{0,-n+p,\rho}\times \mathcal{S}_{0,-n+p+1,\rho}}-C_2\left\|(f,V)\right\|_{\mathcal{H}_{0,-p-2,\rho^{-1}}\times \mathcal{X}_{0,-p-2,\rho^{-1}}}\left\|\left(u,Z\right)\right\|_{\mathcal{H}_{0,-n+p+2,\rho}\times \mathcal{X}_{0,-n+p+2,\rho}} 
\end{split}
\]
and hence, thanks to the basic estimate (in the form of Proposition \ref{pro:basic})
\[
\mathcal{G}(u,Z)\geq C_1 \left\|\left(u,Z\right)\right\|^{2}_{\mathcal{H}_{2,-n+p+2,\rho}\times \mathcal{X}_{1,-n+p+2,\rho}}-C_2\left\|(f,V)\right\|_{\mathcal{H}_{0,-p-2,\rho^{-1}}\times \mathcal{X}_{0,-p-2,\rho^{-1}}}\left\|\left(u,Z\right)\right\|_{\mathcal{H}_{2,-n+p+2,\rho}\times \mathcal{X}_{1,-n+p+2,\rho}}
\]
which immediately implies that $\mathcal{G}\left(\cdot,\cdot\cdot\right)$ is coercive and thus the claim follows.
As a result, we can pick a minimizing sequence $(u_{i}, Z_{i})_{i\in\mathbb{N}}$ which is bounded in $\mathcal{H}_{2,-n+p+2,\rho}\times \mathcal{X}_{1, -n+p+2,\rho}$ (in fact, the previous estimate shows that \textsl{any} minimizing sequence has to be bounded): by Banach-Alaoglu there will be a subsequence which weakly converges to a limit point $(\tu,\tZ)$ and by (weak) lower semicontinuity of the functional we conclude that
\[\mathcal{G}(\tu,\tZ)\leq \liminf_{i\to\infty}\mathcal{G}(u_{i},Z_{i}).
\]
This precisely means that $(\tu,\tZ)$ minimizes the value of $\mathcal{G}$. Finally, the uniqueness statement follows by strict convexity of the functional: indeed, if we had two minima $(u_{1}, Z_{1})$ and $(u_{2}, Z_{2})$ then because of the identity
\[ \mathcal{G}\left(\frac{\left(u_{1}, Z_{1}\right)+\left(u_{2},Z_{2}\right)}{2}\right) =\frac{1}{2}\mathcal{G}\left(u_{1}, Z_{1}\right)+\frac{1}{2}\mathcal{G}\left(u_{2}, Z_{2}\right)-\frac{1}{8}\left\|d\Phi^{\ast}_{\left(g,\pi\right)}\left[u_{2}-u_{1}, Z_{2}-Z_{1}\right]\right\|^{2}_{\mathcal{M}_{0,-n+p+2,\rho}\times \mathcal{S}_{0,-n+p+2,\rho}}
\]
we would reach a contradiction unless $d\Phi^{\ast}_{(g,\pi)}[u_{2}-u_{1}, Z_{2}-Z_{1}]=0$ and by the basic estimate this forces $u_{1}=u_{2}$ as well as $Z_{1}=Z_{2}$ which is what we had to prove.
\end{proof}

\section{The Picard scheme}\label{sec:nonlin}

\subsection{Iterative solution}

In this subsection, we define Banach spaces $X_{1}, X_{2}$ so that the solution operator associated to the linearized problem is in fact a bounded operator $S:X_{1}\to X_{2}$. As a result, we will solve the nonlinear problem iteratively, by following a Picard-type scheme.

Given data $(f^{\ast},V^{\ast})$ where $f^{\ast}$ is a (scalar) function and $V^{\ast}$ is a vector field, we want to find $(\hat{g},\hat{\pi})$ satisfying $\Phi(\hat{g},\hat{\pi})=(f^{\ast},V^{\ast})$. This problem can be more conveniently written as
\[ \Phi(g_{0},\pi_{0})+d\Phi_{(g_{0},\pi_{0})}[h,\omega]+Q_{(g_{0},\pi_{0})}[h,\omega]=(f^{\ast},V^{\ast})
\]
with $(g_{0},\pi_{0})=(g,\pi)$ the data we start with (in our case, they are gotten by rough patching in $\Omega$ of the initial data we are given and the trivial data $(\delta,0)$).
The previous equation takes the form
\[d\Phi_{(g_{0},\pi_{0})}[h,\omega]+Q_{(g_{0},\pi_{0})}[h,\omega]=(f^{\ast},V^{\ast})-\Phi(g_{0},\pi_{0})
\]
and we claim that, in order to solve it, it is sufficient to prove that the quadratic error term decays in the iterative scheme. Throughout this section, we let $\left\|\cdot\right\|_{1}$ denote the norm on the Banach space $X_{1}$ and $\left\|\cdot\right\|_{2}$ denote the norm on the Banach space $X_{2}$ (to be defined in the sequel, based on the form of the local elliptic estimates we have for critical points of the functional $\mathcal{G}$). Their explicit expression is given in equation \eqref{norm1} and \eqref{norm2}, respectively.

\begin{proposition} \label{pro:quadratic} Given any $\lambda>0$, there exists $r_0>0$ sufficiently small so that 
if $\|\left(f_1,V_{1}\right)\|_1<r_0$ and $\|\left(f_2,V_{2}\right)\|_1<r_0$ and we let $\left(h_1,\omega_{1}\right)=S(f_1,V_{1})$, $\left(h_2,\omega_{2}\right)=S(f_2,V_{2})$ then we have
\[ \|Q_{(g,\pi)}\left[h_1,\omega_{1}\right]-Q_{(g,\pi)}\left[h_2,\omega_{2}\right]\|_1\leq \lambda \|\left(h_1,\omega_{1}\right)-\left(h_2,\omega_{2}\right)\|_2.
\] 
\end{proposition}

Once this is proven, the conclusion (in terms of existence, boundary regularity of the gluing and decay at infinity) follows at once:

\begin{theorem} \label{picard} Given $(f,V)\in X_1$ sufficiently small, there is a small
$(h,\omega)\in X_2$ satisfying 
\[ d\Phi_{(g_{0},\pi_{0})}[h,\omega]+Q_{(g_{0},\pi_{0})}[h,\omega]=(f,V).
\]
\end{theorem}

\begin{proof} Assume that $\|\left(f,V\right)\|_{1}<\delta_0$ (with $\delta_{0}$ a small constant to be fixed later in the proof), let $h_0=0, \omega_{0}=0$ and $f_0=0, V_{0}=0$, and we
inductively construct sequences $\left(f_i, V_{i}\right)$ and $(h_i,\omega_{i})$ for $i\geq 1$ such that
\[ d\Phi_{(g,\pi)}\left[h_i,\omega_{i}\right]=\left(f_i, V_{i}\right)\ \mbox{where}\ \left(f_i,V_{i}\right)=-Q_{(g,\pi)}\left[h_{i-1},\omega_{i-1}\right]+\left(f, V\right).
\] 
For $i\geq 1$ we have
\[ d\Phi_{(g,\pi)}([h_{i+1},\omega_{i+1}]-[h_i,\omega_{i}])=(f_{i+1},V_{i+1})-(f_{i},V_{i})=Q_{(g,\pi)}[h_{i-1},\omega_{i-1}]-Q_{(g,\pi)}[h_i,\omega_{i}],
\]
and so by Proposition \ref{pro:quadratic} 
\[\begin{split} \|(f_{i+1},V_{i+1})-(f_i,V_{i})\|_1=\|Q_{(g,\pi)}[h_i,\omega_{i}]-Q_{(g,\pi)}[h_{i-1},\omega_{i-1}]\|_1\leq \lambda\|(h_i,\omega_{i})-(h_{i-1},\omega_{i-1})\|_2 \\ \leq C\lambda\|(f_i,V_{i})-(f_{i-1},V_{i-1})\|_1
\end{split}
\]
where $\lambda$ can be chosen as small as we wish and $C$ is the continuity constant of the solution operator $S$. Let then $r_0$ be small enough so that
$C\lambda<1/2$ in Proposition \ref{pro:quadratic}. We may then iterate this scheme provided that
$\|(f_i,V_{i})\|_1\leq r_0$ for $i=1,\ldots, k$ and in that case we obtain 
\[ \|(f_{k+1},V_{k+1})-(f_k,V_{k})\|_1\leq 2^{-k}\|(f_1,V_{1})-(f_0,V_{0})\|_1=2^{-k}\|(f,V)\|_1< 2^{-k}\delta_0.
\]
From the triangle inequality we then have for any $k$
\[ \|(f_{k+1},V_{k+1})-(f,V)\|_1\leq \sum_{i=1}^{k}2^{-i}\delta_0<2\delta_0,
\]
so if we choose $\delta_0=r_0/4$ we have
\[ \|(f_{k+1},V_{k+1})\|_1\leq \|(f_{k+1},V_{k+1})-(f,V)\|_1+\|(f,V)\|_1<3\delta_0<r_0
\]
for each $k$. We can then iterate indefinitely and the sequence $\{(f_i,V_{i})\}$ is Cauchy
as is $\{(h_i,\omega_{i})\}$ since $S$ is a bounded operator. As a consequence, the sequence $\{(h_i,\omega_{i})\}$ converges in $X_{2}$ to a limit $(h,\omega)$ which satisfies
the equation $d\Phi_{(g,\pi)}[h,\omega]+Q_{(g,\pi)}[h,\omega]=(f,V)$. This completes the proof.
\end{proof}

\begin{remark}\label{rem:small}
We shall explicitly observe that given $(\cg,\ck)$ as in the statement of Theorem \ref{thm:main} and performed the rough patch construction as described in Subsection \ref{subs:doubly}, then 
\[
\lim_{|a|\to\infty}\left\|(\Phi^{(1)}(g,\pi),\Phi^{(2)}(g,\pi))\right\|_1=0
\]
provided the angular cut-off function has sufficiently rapid decay at the boundary of the gluing region $\partial\Omega$. This implies that Theorem \ref{picard} can be legitimately applied (namely: the iteration scheme can indeed be started) and, furthermore, this ensures that the corresponding solution of the non-linear problem $(\hat{g},\hat{k})$ will have small norm in the Banach space $X_2$, in fact as small as we wish provided we pick $|a|$ large enough.
\end{remark}

Therefore, the rest of this section is devoted to the proof of Proposition \ref{pro:quadratic}.

\subsection{Integral estimates}

In this section, we take care of the Sobolev part of the norms defining the Banach spaces $X_{1}$ and $X_{2}$. Given data $(f,V)\in X_{1}$, recall that we have let $(h,\omega)=S(f,V)\in X_{2}$ be the solution of the linearized constraints defined by Proposition \ref{pro:existence}.

\begin{lem}\label{lem:zero} The following bound holds:
\[ \left\|\left(h,\omega\right)\right\|_{\mathcal{M}_{0,-p,\rho^{-1}}\times \mathcal{S}_{0,-p-1,\rho^{-1}}}\leq C\left\|\left(f,V\right)\right\|_{\mathcal{H}_{0,-p-2,\rho^{-1}}\times \mathcal{X}_{0,-p-2,\rho^{-1}}}. \]
\end{lem}

\begin{proof}
Let us recall, from Subsection \ref{subs:varframe}, that given the Euler-Lagrange equation of the functional $\mathcal{G}$ the tensors $h$ and $\omega$ have been defined by means of the equations
\[ h=r^{n-2p}\rho \left(d\Phi^{\ast}_{(g,\pi)}\right)^{(1)}[\tu,\tZ], \ \ \omega=r^{n-2p-2}\rho \left(d\Phi^{\ast}_{(g,\pi)}\right)^{(2)}[\tu,\tZ] 
\]
for $(\tu,\tZ)$ the unique minimizer of $\mathcal{G}$ over the functional space $\mathcal{H}_{2,-n+p+2,\rho}\times \mathcal{X}_{1,-n+p+2,\rho}$. As a result, one has $\mathcal{G}(\tu,\tZ)\leq\mathcal{G}(0,0)=0$ which means
\[\left\|d\Phi_{(g,\pi)}^{\ast}[\tu,\tZ]\right\|^{2}_{\mathcal{M}_{0,-n+p,\rho}\times \mathcal{S}_{0,-n+p+1,\rho}}\leq 2\left\|(\tu,\tZ)\right\|_{\mathcal{H}_{0,-n+p+2,\rho}\times \mathcal{X}_{0,-n+p+2,\rho}}\left\|\left(f,V\right)\right\|_{\mathcal{H}_{0,-p-2,\rho^{-1}}\times \mathcal{X}_{0,-p-2,\rho^{-1}}}.
\]
It follows that, thanks to the basic estimate (Proposition \ref{pro:basic}) we get
\[\left\|d\Phi^{\ast}_{(g,\pi)}[\tu,\tZ]\right\|_{\mathcal{M}_{0,-n+p,\rho}\times \mathcal{S}_{0,-n+p+1,\rho}}\leq C\left\|\left(f,V\right)\right\|_{\mathcal{H}_{0,-p-2,\rho^{-1}}\times \mathcal{X}_{0,-p-2,\rho^{-1}}}
\]
which is equivalent to 
\[ \left\|\left(h,\omega\right)\right\|_{\mathcal{M}_{0,-p,\rho^{-1}}\times \mathcal{S}_{0,-p-1,\rho^{-1}}}\leq C\left\|\left(f,V\right)\right\|_{\mathcal{H}_{0,-p-2,\rho^{-1}}\times \mathcal{X}_{0,-p-2,\rho^{-1}}} 
\]
that is what we had to prove.
\end{proof}

\subsection{The weighted constraints system}

For the Schauder estimates, it is necessary to compute the partial differential equations solved by $(\tu,\tZ)$. Most importantly, we need to prove its ellipticity and determine the rate of decay of its coefficients.
By making use of the very definitions of the derivative maps $d\Phi$ and $d\Phi^{\ast}$, via tedious but elementary computations one can check that the differential system solved by $(\tu,\tZ)$ takes the form
\begin{displaymath}
\begin{cases} \left[T_{1}\tu+\sum_{0\leq\left|\beta\right|\leq 3} a^{(1)}_{\beta}\partial^{\beta}\tu\right]+C^{s}_{(-p-1)}\ast\left[\sum_{0\leq\beta\leq 3}c^{(1)}_{\beta}\partial^{\beta}\tZ\right]=r^{2p-n}\phi^{-2N}f \\
\\
 \left[T_{2}\tZ+\sum_{0\leq\left|\beta\right|\leq 1} c^{(2)}_{\beta}\partial^{\beta}\tZ\right]+C^{s}_{(-p-1)}\ast C^{r}_{(2)}\ast\left[\sum_{0\leq\beta\leq 3}a^{(2)}_{\beta}\partial^{\beta}\tu\right]=-2r^{2p+2-n}\phi^{-2N}V
\end{cases}\end{displaymath}
with
\[ T_{1}=\Delta\left(\Delta \tu\right), \ \ |a^{(1)}_{\beta}|\lesssim r^{\left|\beta\right|-4}\phi^{-2\vee \left|\beta\right|-4}, \ \ |a^{(2)}_{\beta}|\lesssim r^{\left|\beta\right|-3}\phi^{-2 \vee \left|\beta\right|-3}
\]
and
\[ T_{2}=\Delta \tZ+ Div(\nabla \tZ)+C^{s}_{(-2p)}\ast\partial^{2}\tZ, \ \ |c^{(1)}_{\beta}|\lesssim r^{\left|\beta\right|-3}\phi^{-2 \vee \left|\beta\right|-3},\ \ |c^{(2)}_{\beta}|\lesssim r^{\left|\beta\right|-2}\phi^{\left|\beta\right|-2}. 
\]
We stress that in the previous equations all differential operators are Euclidean, namely referred to the flat background metric. We shall also remind the reader that the symbols $C^{s}_{(-q)}$ and $C^{r}_{(-q)}$ have been defined in Subsection \ref{subs:doubly}.

By letting $u=r\tu, \ Z=\tZ$ we can rewrite the previous equations in the final form of the elliptic system for $(u,Z)$ that follows:
\begin{equation}\label{system}
\begin{cases}
\left[T_{1}u+\sum_{0\leq\left|\beta\right|\leq 3} a^{(1)}_{\beta}\partial^{\beta}u\right]+C^{s}_{(-p)}\ast\left[\sum_{0\leq\beta\leq 3}c^{(1)}_{\beta}\partial^{\beta}Z\right]=r^{2p+1-n}\phi^{-2N}f \\
\\
\left[T_{2}Z+\sum_{0\leq\left|\beta\right|\leq 1} c^{(2)}_{\beta}\partial^{\beta}Z\right]+C^{s}_{(-p)}\ast\left[\sum_{0\leq\beta\leq 3}a^{(2)}_{\beta}\partial^{\beta}u\right]=-2r^{2p+2-n}\phi^{-2N}V
\end{cases}
\end{equation}
where the coeffiecients are not necessarily the same as above but satisfy all of the same decay estimates in $(r,\phi)$.

\subsection{Douglis-Nirenberg ellipticity}

In order to proceed further, we need to prove H\"older estimates on the solution $(h,\omega)$. We make use of a specific result for inhomogeneous systems, due to Douglis-Nirenberg \cite{DN55}, which we briefly recall here for the convenience of the reader. To that aim, we first need to introduce some notation.

Let us consider a system of linear partial differential equations of the form
\begin{equation}\label{diffsys}
L_{i}w=\sum_{j=1}^{N}l_{ij}w_{j}=f_{i}, \ \ i=1,\ldots, N
\end{equation}
where for any $j=1,\dots, N$ we have that $w_{j}$ is a function of $n$ variables $(x_{1}, \ldots, x_{n})$ with $x\in\Gamma$ some regular domain of the Euclidean space. Let us assume that each differential operator $l_{ij}$ can be expressed as a polynomial in $\frac{\partial}{\partial x_{1}},\ldots, \frac{\partial}{\partial x_{n}}$ with sufficiently smooth coefficients.
Moreover, let us suppose that there exist $2N$ integers $\left\{s_{1},\ldots, s_{N}, t_{1},\ldots, t_{N}\right\}$ so that $l_{ij}$ has order \textsl{less or equal} than $s_{i}+t_{j}$ and let $l_{ij}'$ be the sum of the terms of $l_{ij}$ having order \textsl{exactly equal} to $s_{i}+t_{j}$ for any choice of our indices. Of course, it is not the case that such numbers always exist as one is supposed to find \textsl{integer} solutions to a linear system of $N^{2}$ equations in $2N$ unknowns: however, when this does happen, the determinant of the characteristic matrix of \eqref{diffsys}, namely $\left(l'_{ij}(x,\xi)\right)_{1\leq i,j\leq n}$ is an \textsl{homogeneous} polynomial $P(x,\xi)$ of degree $m=\sum_{k=1}^{N}(s_{k}+t_{k})$ in $\xi_{1},\ldots,\xi_{n}\in\mathbb{R}$. Indeed, each summand in $P(x,\xi)$ will have degree of the form $\sum_{k=1}^{N}(s_{k}+t_{\sigma(k)})$ for some $\sigma\in\mathcal{S}_{N}$, the symmetric group on $N$ elements, and of course
\[\sum_{k=1}^{N}(s_{k}+t_{\sigma\left(k\right)})=\sum_{k=1}^{N}s_{k}+\sum_{k=1}^{N}t_{\sigma(k)}=\sum_{k=1}^{N}(s_{k}+t_{k})=m.
\]

We will say that the system \eqref{diffsys} is \textsl{elliptic} (according to Douglis-Nirenberg) if there exist $s_{1},\ldots, s_{N},t_{1},\ldots,t_{N}\in\mathbb{Z}$ so that, at every point $x$ the determinant $P(x,\xi)$ does not vanish for every $\xi\neq 0$.  

Obviously, when we deal with an elliptic system as above, we can always reduce to the case when
\[ s_{i}\leq 0, \ i=1,\ldots, N \ \ \  \max_{i}s_{i}=0 \ \ \  t_{j}\geq 0, \ j=1,\ldots, N 
\]
which is motivated by the form of the Schauder estimates for a single (scalar) PDE, as will be apparent from the statement below.
Correspondingly, let us set
\[ \min_{i}s_{i}=-s, \ \ \ \max_{j}t_{j}=t.
\]

Now, assume the domain $\Gamma$ is bounded and let $d:\Gamma\to\mathbb{R}$ be the distance function from the boundary $\partial\Gamma$: for every $k\in\mathbb{Z}_{\geq 0}$ and $l\in\mathbb{R}$ we consider the weighted H\"older norm
\[ \left\|u\right\|^{\left(l\right)}_{k,\alpha}=\sum_{i=0}^{k}\sup_{x\in\Gamma}d(x)^{-l+i}\left|\partial^{i}u(x)\right|+\sup_{x\in\Gamma}d(x)^{-l+k+\alpha}\left[\partial^{k}u\right]_{\alpha}
\] 
and let $\mathcal{C}^{k,\alpha}_{l}(\Gamma;\mathbb{R})$ be the Banach space which is gotten by completing the space of restrictions of elements in $\mathcal{C}_{c}^{\infty}(\mathbb{R}^{n};\mathbb{R})$ with respect to such norm.
In order to state the interior regularity theorem of Douglis-Nirenberg, we need to give the following:

\

\textsl{Hypothesis (\textbf{H}):} let us write $l_{ij}(x,\partial)=\sum_{|\beta|=0}^{s_{i}+t_{j}}a_{ij,\beta}(x)\partial^{\beta}$ where the sum is understood to be first taken over all terms of $l_{ij}$ of order equal to $|\beta|$. For some constant $\alpha\in\left(0,1\right)$, a fixed positive constant $K$ and all the indices $i,j\in\left\{1,\ldots, n\right\}$ we require that:
\begin{enumerate}
\item{the coefficients $a_{ij,\beta}$ belong to the space $\mathcal{C}^{-s_{i},\alpha}_{-s_{i}-t_{j}+|\beta|}$ and 
\[ \sup_{i,j,\beta}\left\|a_{ij,\beta}\right\|_{-s_{i},\alpha}^{\left(-s_{i}-t_{j}+|\beta|\right)}\leq K; 
\]}
\item{the inhomogeneous term $f_{i}$ belongs to the space $\mathcal{C}^{-s_{i},\alpha}_{-s_{i}-t}$;}
\item{the characteristic determinant satisfies
\[ P(x,\xi)\geq K^{-1}\left(\sum_{i=1}^{n}\xi_{i}^{2}\right)^{m/2}.
\]
}
\end{enumerate}
Let us explicitly remark that the notation we are using here for weighted H\"older spaces is different from that in \cite{DN55} and the two are patently incompatible.

\begin{theorem}\label{regsys}(see \cite{DN55}, Theorem 1) Let $u$ be a solution of the system \eqref{diffsys} under the assumption (\textbf{H}). Assume that $u_{j}\in \mathcal{C}^{0,0}_{-t+t_{j}}$ and that $u_{j}$ has H\"older continuous derivatives up to order $t_{j}$ in $\Gamma$ for each value of the index $j=1,\ldots, N$. Then $u_{j}\in \mathcal{C}^{t_{j},\alpha}_{-t+t_{j}}$ and
\[ \left\|u_{j}\right\|_{t_{j},\alpha}^{\left(-t+t_{j}\right)}\leq C\left(\sum_{i=1}^{N}\left\|u_{i}\right\|_{0,0}^{(-t+t_{i})}+\sum_{i=1}^{N}\left\|f_{i}\right\|^{(-s_{i}-t)}_{-s_{i},\alpha}\right)
\]
for some constant $C=C(K,n,N,s_{1},\ldots, t_{N},\alpha)$.
\end{theorem}

Let us now discuss the applicability of this result to our problem, namely to the system \eqref{system}. 
If $\textrm{dim}(M)=n$ one first needs to find $2n+2$ integers $s_{1},\ldots,s_{n+1}, t_{1},\ldots, t_{n+1}$ satisfying the algebraic system
\[\begin{cases}
s_{1}+t_{1}=4 \\
s_{i}+t_{1}=3 & i>1 \\
s_{1}+t_{j}=3 & j>1 \\
s_{i}+t_{j}=2 & i,j>1 \\
\end{cases}
\]
and we will pick 
\[ s_{1}=0, \ \ s_{i}=-1, \ \ t_{1}=4, \ \ t_{j}=3 \ \ \ \  \ \ 2\leq i,j\leq n+1.
\]

\begin{lem}\label{lem:ellsys}
There exists $a_{\infty,N}\in\mathbb{R}$ such that for any $a\in\mathbb{R}^{n}$ such that $\left|a\right|\geq a_{\infty,N}$ the system \eqref{system} is elliptic, in the sense of Douglis-Nirenberg, on the domain $\Omega=\Omega(a)$.
\end{lem}

\begin{proof}
By the very definition of determinant, we can write the characteristic determinant for the Einstein constraint \eqref{system} as
\[ 
P(x,\xi)\geq \tilde{P}(x,\xi)-Cs(x)^{-p}\left|\xi\right|^{m}
\]
for $m=2n+4$ and $C$ a constant which only depends on the decay assumptions on $g-\delta$ and $\pi$, where we have set
\begin{equation*}
\tilde{P}(x,\xi)=\det
\begin{pmatrix}
|\xi|^4      &     0              &     0             & \cdots    & 0 \\
    0        & |\xi|^2 + \xi_1^2  & \xi_1 \xi_2       & \cdots    & \xi_1 \xi_n  \\
		0        & \xi_1 \xi_2        & |\xi|^2 + \xi_2^2 & \cdots    & \xi_2 \xi_n  \\
		\vdots   &     \vdots         & \vdots            & \ddots    & \vdots       \\
		0        & \xi_1 \xi_n        & \xi_2 \xi_n       & \cdots    & |\xi|^2 + \xi_n^2 \\
\end{pmatrix}.
\end{equation*}
Let us notice that this would be the characteristic matrix of the Einstein system in the case when $(\tu,\tZ)$ were not coupled, namely in the time-symmetric case when $\pi=0$.
Now, the block structure of such matrix immediately implies that
\[ \tilde{P}(x,\xi)=\left|\xi\right|^{4}\det\left(A+bb^{t}\right)
\]
for 
\begin{equation*}
A(\xi)=\left|\xi\right|^{2}
\begin{pmatrix}
    1        &     0              &     0             & \cdots    & 0 \\
    0        &     1              &                    & \cdots    & 0  \\
		0        &     0              &     1 & \cdots    & 0  \\
		\vdots   &     \vdots         & \vdots            & \ddots    & \vdots       \\
		0        &     0              & 0       & \cdots    & 1 \\
\end{pmatrix}, \  \
b=
\begin{pmatrix}
\xi_{1} \\
\xi_{2} \\
\vdots  \\
\vdots  \\
\xi_{n}  \\
\end{pmatrix}.
\end{equation*}
It is a well-known fact in linear algebra that, given an $n\times n$ matrix $A$ and vectors $b_{1}, b_{2}\in\mathbb{R}^{n}$ then $\det(A+b_{1}b_{2}^{t})=\det(A)+b_{2}^{t}\textrm{adj}(A)b_{1}$ and thus, in our case this implies at once that in fact $\tilde{P}(x,\xi)=2\left|\xi\right|^{2n+4}$. Hence
\[ P(x,\xi)\geq \left(2-Cs(x)^{-p}\right)\left|\xi\right|^{m}, \ \ \ x\in \Omega \ \ \xi\in\mathbb{R}^{n}
\]
and the conclusion follows by picking $|a|$ large enough by virtue of inequality \eqref{eq:min}.
\end{proof}

Before proceeding further, let us make an important remark on our notations: \textsl{in the rest of this subsection, as well as in the next one we will adopt the notation $\mathcal{H}_{k,q,\rho}$ when referring to any type of tensor.}

\begin{lem}\label{pwsch}
For any $\alpha\in\left(0,1\right)$ the following H\"older bounds hold true for $(u,Z)$:
\[\left\|u\right\|_{4,\alpha}^{(0,0)}\leq C\left(\left\|\overline{u}\right\|_{0,0}^{(0,0)}+\left\|\overline{Z}\right\|_{0,0}^{(-1,-1)}+\left\|f\right\|^{\left(-2p-5+n, -4+2N\right)}_{0,\alpha}+\left\|V\right\|^{\left(-2p-5+n,-3+2N\right)}_{1,\alpha}\right)
\]
\[\left\|Z\right\|_{3,\alpha}^{(-1,-1)}\leq C\left(\left\|\overline{u}\right\|_{0,0}^{(0,0)}+\left\|\overline{Z}\right\|_{0,0}^{(-1,-1)}+\left\|f\right\|^{\left(-2p-5+n, -4+2N\right)}_{0,\alpha}+\left\|V\right\|^{\left(-2p-5+n,-3+2N\right)}_{1,\alpha}\right).
\]
as well as the following for the solution $(h,\omega)$ of the linearized problem:
\[\left\|h\right\|^{\left(-p,+N-n/2-2\right)}_{2,\alpha}\leq C\left(\left\|(f,V)\right\|_{\mathcal{H}_{0,-p-2,\rho^{-1}}\times \mathcal{H}_{0,-p-2,\rho^{-1}}}+\left\|f\right\|^{\left(-p-2,N-n/2-4\right)}_{0,\alpha}+\left\|V\right\|^{(-p-2,N-n/2-3)}_{1,\alpha}\right)
\]
\[\left\|\omega\right\|^{\left(-p-1,N-n/2-2\right)}_{2,\alpha}\leq C\left(\left\|(f,V)\right\|_{\mathcal{H}_{0,-p-2,\rho^{-1}}\times \mathcal{H}_{0,-p-2,\rho^{-1}}}+\left\|f\right\|^{\left(-p-2,N-n/2-4\right)}_{0,\alpha}+\left\|V\right\|^{(-p-2,N-n/2-3)}_{1,\alpha}\right).
\]
\end{lem}

Motivated by the previous Lemma, we are now in position to actually define the Banach spaces $X_{1}, X_{2}$. Given $\Omega$, as usual, the (regularized) region between the two cones we consider $X_{i}, \ i=1,2$ to be the closure of the set of smooth $(\textrm{functions}, \textrm{symmetric} \ (0,2)-\textrm{tensors})$ in $\Omega\hookrightarrow\R^{n}$ for which the norm $\left\|\cdot\right\|_{i}$ is finite, where
\begin{equation}\label{norm1} \left\|(f,V)\right\|_{1}=\left\|(f,V)\right\|_{\mathcal{H}_{0,-p-2,\rho^{-1}}\times \mathcal{H}_{0,-p-2,\rho^{-1}}}+\left\|f\right\|^{\left(-p-2,N-n/2-4\right)}_{0,\alpha}+\left\|V\right\|^{(-p-2,N-n/2-3)}_{1,\alpha}
\end{equation}
and
\begin{equation}\label{norm2} \left\|(h,\omega)\right\|_{2}=\left\|(h,\omega)\right\|_{\mathcal{H}_{0,-p,\rho^{-1}}\times \mathcal{H}_{0,-p-1,\rho^{-1}}}+\left\|h\right\|^{\left(-p,N-n/2-2\right)}_{2,\alpha}+\left\|\omega\right\|^{\left(-p-1,N-n/2-2\right)}_{2,\alpha}.
\end{equation}

\begin{proof}
Thanks to Lemma \ref{lem:ellsys} and the decay of the coefficients of our system, we are in position to apply Theorem \ref{regsys} and get
\begin{align*}
& \sum_{i=0}^{4}d(x)^{i}\left|\partial^{i}u(x)\right|+d(x)^{4+\alpha}\left[\partial^{4}u\right]_{\alpha,B_{d(x)/2}\left(x\right)}+ \sum_{i=0}^{3}d(x)^{i+1}\left|\partial^{i}Z(x)\right|+d(x)^{4+\alpha}\left[\partial^{3}Z\right]_{\alpha,B_{d(x)/2}\left(x\right)} \\
& \leq C\left[\overline{u}(x)+d(x)\overline{Z}(x)\right]+Cd(x)^{4}r(x)^{2p+1-n}\phi(x)^{-2N}\left\{\sup_{B_{3d(x)/4}\left(x\right)}\left|f\right|+d(x)^{\alpha}\left[f\right]_{\alpha, B_{3d(x)/4}\left(x\right)}\right\} \\
& + Cd(x)^{3}r(x)^{2p+2-n}\phi(x)^{-2N}\left\{\sup_{B_{3d(x)/4}\left(x\right)}\left|V\right|+d(x)\sup_{B_{3d(x)/4}\left(x\right)}\left|\partial V\right|+d(x)^{1+\alpha}\left[\partial V\right]_{\alpha, B_{3d(x)/4}\left(x\right)}\right\}
\end{align*}
which immediately implies the first two claimed inequalities.
In order to prove the other estimates, one needs the following upper bounds for $\overline{u}$ and $\overline{Z}$. 
For the former:
\begin{align*}
 \left|\overline{u}(x)\right|^{2} & \leq Cd(x)^{-n}\int_{B_{d(x)/2}(x)}\left|u(y)\right|^{2}\,d\mathscr{L}^n(y) \leq Cd(x)^{-n}r(x)^{2}\int_{B_{d(x)/2}(x)}\left|\tu(y)\right|^{2}\,d\mathscr{L}^n(y) \\
& \leq Cr(x)^{2(p+3-n)}\phi(x)^{-2N-n}\left\|\tu\right\|^{2}_{\mathcal{H}_{2,p+2-n,\rho}} \\
& \leq Cr(x)^{2(p+3-n)}\phi(x)^{-2N-n}\left\|(\tu,\tZ)\right\|^{2}_{\mathcal{H}_{2,-n+p+2,\rho}\times \mathcal{H}_{1,-n+p+2,\rho}} \\
& \leq C r(x)^{2(p+3-n)}\phi(x)^{-2N-n}\left\|d\Phi^{\ast}_{\left(g,\pi\right)}(\tu,\tZ)\right\|^{2}_{\mathcal{H}_{0,-n+p,\rho}\times \mathcal{H}_{0,-n+p+1,\rho}} \\
& \leq C r(x)^{2(p+3-n)}\phi(x)^{-2N-n}\left\|\left(f,V\right)\right\|^{2}_{\mathcal{H}_{0,-p-2,\rho^{-1}}\times \mathcal{H}_{0,-p-2,\rho^{-1}}} 
\end{align*}
where we have used both the basic estimate, Proposition \ref{pro:basic}, and Lemma \ref{lem:zero}.
Similarly, for the latter:
\begin{align*}
 \left|d(x)\overline{Z}(x)\right|^{2} & \leq Cr(x)^{2-n}\phi(x)^{2-n}\int_{B_{d(x)/2}(x)}\left|Z(y)\right|^{2}\,d\mathscr{L}^n(y) \leq Cr(x)^{2(p+3-n)}\phi(x)^{-2N-n+2}\left\|\tZ\right\|^{2}_{\mathcal{H}_{1,-n+p+2,\rho}} \\
& \leq Cr(x)^{2(p+3-n)}\phi(x)^{-2N-n+2}\left\|(\tu,\tZ)\right\|^{2}_{\mathcal{H}_{2,-n+p+2,\rho}\times \mathcal{H}_{1,-n+p+2,\rho}} \\
& \leq C r(x)^{2(p+3-n)}\phi(x)^{-2N-n+2}\left\|d\Phi^{\ast}_{\left(g,\pi\right)}(\tu,\tZ)\right\|^{2}_{\mathcal{H}_{0,-n+p,\rho}\times \mathcal{H}_{0,-n+p+1,\rho}} \\
& \leq Cr(x)^{2(p+3-n)}\phi(x)^{-2N-n+2}\left\|\left(f,V\right)\right\|^{2}_{\mathcal{H}_{0,-p-2,\rho^{-1}}\times \mathcal{H}_{0,-p-2,\rho^{-1}}}
\end{align*}
so that by taking square roots one gets:
\[\left\|\overline{u}\right\|_{0,0}^{(0,0)}+\left\|\overline{Z}\right\|_{0,0}^{(-1,-1)}\leq C r(x)^{p+3-n}\phi(x)^{-N-n/2}\left\|\left(f,V\right)\right\|_{\mathcal{H}_{0,-p-2,\rho^{-1}}\times \mathcal{H}_{0,-p-2,\rho^{-1}}}.
\]
At that stage, one can exploit the H\"older estimates above, together with the very definition of the tensors $h$ and $\omega$ in terms of $(u,Z)$ to complete the proof.
For instance, one has for the zeroth order estimate on $h$:
\begin{align*}
 \left|h(x)\right| & \leq Cr(x)^{n-2p-1}\rho(x) d(x)^{-2}\left[d(x)^{2}\left|\partial^{2}u(x)\right|+d(x)\left|\partial u(x)\right|+\left|u(x)\right|\right] \\
& +Cr(x)^{n-2p-1}\rho(x) s(x)^{-p}d(x)^{-2}\left[d(x)^{2}\left|\partial Z(x)\right|+d(x)\left|Z(x)\right|\right] \\
& \leq C r(x)^{n-2p-3}\phi(x)^{2N-2} \times \\
& \left[r^{p+3-n}\phi^{-N-n/2}\left\|(f,V)\right\|_{\mathcal{H}_{0,-p-2,\rho^{-1}}\times \mathcal{H}_{0,-p-2,\rho^{-1}}}+\left\|f\right\|^{(-2p-5+n,-4+2N)}_{0,\alpha}+\left\|V\right\|^{(-2p-5+n,-3+2N)}_{1,\alpha}\right] \\ 
& \leq Cr(x)^{-p}\phi(x)^{N-n/2-2}\left[\left\|(f,V)\right\|_{\mathcal{H}_{0,-p-2,\rho^{-1}}\times \mathcal{H}_{0,-p-2,\rho^{-1}}}+\left\|f\right\|^{(-p-2, N-n/2-4)}_{0,\alpha}+\left\|V\right\|^{(-p-2, N-n/2-3)}_{0,\alpha}\right].
\end{align*} 
The other pointwise and (concerning the second derivatives) H\"older estimates for $h$ and $\omega$ are completely analogous and we omit the elementary details.
\end{proof}

\subsection{Convergence of the iteration}

We now proceed with the proof of Proposition \ref{pro:quadratic}.

\begin{proof} In order for us to obtain estimates for the difference of the quadratic terms, $Q_{(g,\pi)}[h_{1},\omega_{1}]-Q_{(g,\pi)}[h_{2},\omega_{2}]$ it is convenient to find an exact representation formula for such difference. Recalling the expression for the constraint equations, we need to consider 
\[ Q_{(g,\pi)}[h,\omega]=\left(\mathscr{H}(\overline{g},\overline{\pi})-\mathscr{H}(g,\pi)-d\Phi^{(1)}_{(g,\pi)}[h,\omega], Div_{\overline{g}}(\overline{\pi})-Div_{g}(\pi)-d\Phi^{(2)}_{(g,\pi)}[h,\omega]\right)
\]
where $\overline{g}_{ij}=g_{ij}+h_{ij}$ and $\overline{\pi}^{ab}=\pi^{ab}+\omega^{ab}$.
Throughout this subsection, we shall adopt the following simplified notation: $R$ (resp. $\overline{R}$) for the scalar curvature $R_g$ of $g$ (resp $R_{\overline{g}}$ of $\overline{g}$) and $R_{ij}$ (resp. $\overline{R}_{ij}$) for the expression in local coordinates of the Ricci curvature $Ric_g$ of $g$ (resp. $Ric_{\overline{g}}$ of $\overline{g}$).
Concerning the first component, we have to study the nonlinear part of the difference
\[\overline{R}-R+\frac{1}{n-1}\left[\left(Tr_{\overline{g}}\overline{\pi}\right)^{2}-\left(Tr_{g}\pi\right)^{2}\right]-\left[\left|\overline{\pi}\right|_{\overline{g}}^{2}-\left|\pi\right|^{2}_{g}\right]
\]
so let us work out each of the three summands separately.
We start with the scalar curvature term: first from the formula for the Ricci curvature we have
\[ \tR_{ij}-R_{ij}=D^k_{ij;k}-D^k_{ki;j}+D^k_{kl}D^l_{ij}
-D^k_{jl}D^l_{ki}
\]
where $D=\tG-\G$ is the difference of the Levi-Civita connections and
the semi-colon denotes the covariant derivative with respect to $g$.
If we write $\tg=g+h$, then we have
\[ D^k_{ij}=\frac{1}{2}\tg^{kl}(h_{il;j}+h_{jl;i}-h_{ij;l}).
\]
Now taking traces with respect to $\tg$ we have
\[ \tR-\tg^{ij}R_{ij}=\tg^{ij}(\tg^{kl}h_{il;j})_{;k}-\tg^{ij}(\tg^{kl}h_{kl;i})_{;j}+q(\tg,h)
\]
where we use $q(\tg,h)$ to denote a quadratic polynomial in the first covariant derivatives
of $h$ with coefficients depending on $\tg$. Since $g$ is parallel (indeed, it is the background metric we are covariant derivatives with respect of) we have
$\tg_{ij;k}=h_{ij;k}$, and $\tg^{ij}_{;k}=-\tg^{il}\tg^{jm}h_{lm;k}$. We may thus rewrite
the expression
\[ \tR-R=\tg^{ij}\tg^{kl}h_{il;jk}-\tg^{ij}\tg^{kl}h_{kl;ij}+(\tg^{ij}-g^{ij})R_{ij}+
q(\tg,h)
\]
where we have modified $q(\tg,h)$. We note that the linear term $L(h)$  of $\tR-R$ is given by
\[ L(h)=[g^{ij}g^{kl}(h_{il;jk}-h_{kl;ij})-g^{ik}g^{jl}h_{kl}]R_{ij}.
\]
Therefore the quadratic part is
\begin{equation}\label{Q-formula1} \tR-R-L(h)=(\tg^{ij}\tg^{kl}-g^{ij}g^{kl})(h_{il;jk}-h_{kl;ij})+T+q(\tg,h)
\end{equation}
where the term $T$ is given by
\[ T=(\tg^{ij}-g^{ij})R_{ij}+g^{ik}g^{jl}h_{kl}R_{ij}.
\]
We can rewrite $T$ by setting $g_t=g+th$ for $t\in [0,1]$, and writing
\[ \tg^{ij}-g^{ij}=\int_0^1\frac{d}{dt}(g_t^{ij})\ d\mathscr{L}^1=-\left(\int_0^1 g_t^{ik}g_t^{jl}\ d\mathscr{L}^1\right)h_{kl}.
\]
Hence we have 
\begin{equation}\label{T-formula} T=-\int_0^1(g_t^{ik}g_t^{jl}-g^{ik}g^{jl})\ d\mathscr{L}^1\ h_{kl}R_{ij}.
\end{equation}
For the trace term, one can write
\[(Tr_{\tg}\tp)^{2}-(Tr_{g}\pi)^{2}=\left(\tg_{ij}\tp^{ij}+g_{ij}\pi^{ij}\right)\left(\tg_{kl}\tp^{kl}-g_{kl}\pi^{kl}\right)
\]
and by the definitions of $\tg$ and $\tp$
\[\tg_{kl}\tp^{kl}-g_{kl}\pi^{kl}=g_{kl}\omega^{kl}+h_{kl}\pi^{kl}+h_{kl}\omega^{kl} , \ \ \tg_{ij}\tp^{ij}+g_{ij}\pi^{ij}=2g_{ij}\pi^{ij}+g_{ij}\omega^{ij}+h_{ij}\pi^{ij}+h_{ij}\omega^{ij}
\]
so the quadratic part is given by
\begin{equation}\label{Q-formula2}
\left(2g_{ij}\pi^{ij}\right)\left(h_{kl}\omega^{kl}\right)+\left(g_{ij}\omega^{ij}+h_{ij}\pi^{ij}+h_{ij}\omega^{ij}\right)\left(g_{kl}\omega^{kl}+h_{kl}\pi^{kl}+h_{kl}\omega^{kl}\right).
\end{equation}
The squared norm terms can be expanded as follows:
\[\left|\tp\right|^{2}_{g}-\left|\pi\right|^{2}_{g}=\tp^{ik}\tp^{jl}\tg_{ij}\tg_{kl}-\pi^{ik}\pi^{jl}g_{ij}g_{kl}=(\pi^{ik}+\omega^{ik})\left(\pi^{jl}+\omega^{jl}\right)\left(g_{ij}+h_{ij}\right)\left(g_{kl}+h_{kl}\right)-\pi^{ik}\pi^{jl}g_{ij}g_{kl}
\]
so it is readily checked that the quadratic part is given by
\begin{equation}\label{Q-formula3}
\pi^{ik}\pi^{jl}h_{ij}h_{kl}+\omega^{ik}\omega^{jl}g_{ij}g_{kl}+\left(\pi^{ik}\omega^{jl}+\pi^{jl}\omega^{ik}+\omega^{ik}\omega^{jl}\right)\left(g_{ij}h_{kl}+g_{kl}h_{ij}+h_{ij}h_{kl}\right).
\end{equation}
Now we observe that from \eqref{Q-formula1}, \eqref{T-formula}, \eqref{Q-formula2} and \eqref{Q-formula3} the term $Q^{(1)}_{(g,\pi)}([h,\omega])$ is a
sum of terms of the form
\[ E_{1}(h)(\partial^2 h),\ E_{2}(h)(h)(\Gamma^2+\partial\Gamma+\pi^{2}), \  E_3(h)(\partial h)(\Gamma)  , \ E_{4}(h)(\partial h)(\partial h), \ E_{5}(h)(\omega)(\omega), \ E_{6}(h)(h)(\omega)(\pi) 
\]
where $E_{1},E_{2},E_{3}, E_{4}, E_{5}, E_{6}$ are smooth coefficient systems depending on $h$ with $E_i(0)=0$ for $i=1,2,3$.

Now, let us concern ourselves with the second component of the constraints. Let us remark that $Div_{g}(\pi)$ is of course linear when the background metric is fixed $(g)$ and we let the differential operator $Div_{g}$ act on the symmetric tensor $\pi$, while $Div_{g}\pi$ is not linear as a function of the couple $(g,\pi)$ and this is the reason why the following computation is not trivial.
Indeed, since we need to deal with two different background metrics ($g$ and $\tg$) we will start by unwinding the covariant derivatives:
\[ \left(Div_{g}\pi\right)^{l}=(\pi^{il}_{,i}+\Gamma^{i}_{ik}\pi^{kl}+\Gamma^{l}_{ik}\pi^{ik})
\]
and thus, for the difference we can write
\begin{align*}
\left(Div_{\tg}\tp\right)^{l} & -\left(Div_{g}\pi\right)^{l}=\op^{il}_{,i}-\pi^{il}_{,i}+\oG^{i}_{ik}\op^{kl}-\G^{i}_{ik}\pi^{kl}+\oG^{l}_{ik}\op^{ki}-\Gamma^{l}_{ik}\pi^{ki} \\
& =(\op^{il}_{,i}-\pi^{il}_{i})+D^{i}_{ik}\op^{kl}+\Gamma^{i}_{ik}(\op^{kl}-\pi^{kl})+D^{l}_{ik}\op^{ki}+\G^{l}_{ik}(\op^{ki}-\pi^{ki}) \\
& =\omega^{il}_{,i}+D^{i}_{ik}(\pi^{kl}+\omega^{kl})+\Gamma^{i}_{ik}\omega^{kl}+D^{l}_{ik}(\pi^{ki}+\omega^{ki})+\Gamma^{l}_{ik}\omega^{ki}.
\end{align*}

As a result, making use of the expression for the difference of Christoffel symbols that has been derived above, namely
\[ D_{ij}^{k}=\frac{1}{2}\left(g^{kl}-\int_{0}^{1}g_{t}^{ak}g_{t}^{bl}h_{ab}\,d\mathscr{L}^1\right)(h_{il;j}+h_{jl;i}-h_{ij;l})
\]
we can express the quadratic part of the second component as as finite sum of terms that belong to one of the following categories:
\[  E_{7}(h)(\partial h)(\pi), \  E_8(h)(h)(\Gamma\pi), \  E_{9}(h)(\partial h)(\omega), \ E_{10}(h)(h)(\omega)(\Gamma) 
\]
where $E_{7}, E_{8}, E_9, E_{10}$ are smooth coefficient systems depending on ($g$ and) $h$, moreover $E_{7}(0)=0$ and $E_8(0)=0$.

These preliminaries being done, we can easily get pointwise upper bounds for the first and second components of $Q_{(g,\pi)}\left[h_{1},\omega_{1}\right]-Q_{(g,\pi)}\left[h_{2},\omega_{2}\right]$.
Concerning the first component, we will have (for small $[h_{1},\omega_{1}]$ and $[h_{2}, \omega_{2}]$ in the $X_{2}$-topology)
\begin{align*} &|Q^{(1)}_{(g,\pi)}[h_1,\omega_{1}]-Q^{(1)}_{(g,\pi)}[h_2,\omega_{2}]|\leq C(|h_1|+|h_2|)|\partial^2(h_1-h_2)| \\ &
+C(|\partial h_1|+|\partial h_2|+|\Gamma|(|h_1|+|h_2|))|\partial (h_1-h_2)|\\ &
+C(|\partial^2 h_1|+|\partial^2 h_2|+|\partial h_1|^2+|\partial h_2|^2+|\Gamma|(|\partial h_1|+|\partial h_2|)) h_1-h_2|\\ &
+C((\left|\partial\Gamma\right|+|\Gamma|^2+\left|\pi\right|^{2})(|h_1|+|h_2|)+\left|\pi\right|\left(\left|\omega_{1}\right|+\left|\omega_{2}\right|\right))) |h_1-h_2| \\
 &+C\left(|\pi|\left(\left|h_{1}\right|+\left|h_{2}\right|\right)+\left(\left|\omega_{1}\right|+\left|\omega_{2}\right|\right)\right)\left|\omega_{1}-\omega_{2}\right|.
\end{align*}
Instead, for the second component, we have:
\begin{align*}
&|Q^{(2)}_{(g,\pi)}[h_{1},\omega_{1}]-Q^{(2)}_{(g,\pi)}[h_2,\omega_{2}]|\leq C\left(\left|\pi\right|\left(\left|h_{1}\right|+\left|h_{2}\right|\right)+\left(\left|\omega_{1}\right|+\left|\omega_{2}\right|\right)\right)\left|\partial \left(h_{1}- h_{2}\right)\right|  \\ &
+C((\left|\pi\right|+|\omega_{1}|+|\omega_{2}|)(\left|\partial h_{1}\right|+\left|\partial h_{2}\right|))+|\Gamma|(|\pi|(h_1|+|h_2|)+(|\omega_1|+|\omega_2|)))\left|h_{1}-h_{2}\right| \\
&+C(\left|\partial h_{1}\right|+\left|\partial h_{2}\right|+|\Gamma|(|h_1|+|h_2|))\left|\omega_{1}-\omega_{2}\right|.
\end{align*}
Since we are assuming that $\|(f_1, V_{1})\|_{1}< r_0$ and $\|(f_2, V_{2})\|_{1}<r_0$, our Schauder estimates (Lemma \ref{pwsch}) imply that for $0\leq j\leq 2$
\[ |\partial^j h_a|\leq Cr_0r^{-p-j}\phi^{N-n/2-2-j}\leq Cr_0r^{-p-j}\phi^{N-n/2-4}
\]
as well as
\[\left|\partial^{j} \omega_{a}\right|\leq Cr_{0}r^{-p-1-j}\phi^{N-n/2-2-j}\leq Cr_0r^{-p-1-j}\phi^{N-n/2-4}
\]
for $a=1,2$.

It follows that, due to our decay assumptions on the data $g-\delta, \pi$ one obtains the pointwise bounds
\[|Q^{(i)}_{(g,\pi)}[h_{1},\omega_{1}]-Q^{(i)}_{(g,\pi)}[h_2,\omega_{2}]|\leq Cr_{0}r^{-2p-2}\phi^{2(N-n/2-4)}\left\|(h_{1}-h_{2}, \omega_{1}-\omega_{2})\right\|_{2}, \ \ x\in\Omega \ \  i=1,2
\]
both for the first and for the second component.
As a result, we have
\[ \int_\Omega |Q^{(i)}_{(g,\pi)}[h_1,\omega_{1}]-Q^{(i)}_{(g,\pi)}[h_2,\omega_{2}]|^2r^{-n+2p+4}\phi^{-2N}\ d\mathscr{L}^n \leq Cr_0^2\left\|(h_{1}-h_{2}, \omega_{1}-\omega_{2})\right\|^{2}_{2} \{\int_\Omega r^{-n-2p}\phi^{2N-2n-16} d\mathscr{L}^n\}.
\]
and thus, since for any $p>0$ and $N$ large enough (say $N>n+8$) the above integral is of course finite, one concludes
\[ \left\|Q_{(g,\pi)}[h_{1},\omega_{1}]-Q_{(g,\pi)}[h_{2},\omega_{2}]\right\|_{\mathcal{H}_{0,-p-2,\rho^{-1}}\times \mathcal{H}_{0,-p-2,\rho^{-1}}}\leq Cr_{0}\left\|(h_{1},\omega_{1})-\left(h_{2}, \omega_{2}\right)\right\|_2
\]
which is the first assertion we had to prove, based on the definition of the norm $\left\|\cdot\right\|_{1}$.

We can then proceed with the pointwise estimates for the H\"older seminorms of $Q^{(1)}_{(g,\pi)}$.

As a preliminary remark which will be used several times in the sequel of this proof, we observe that for two functions $u_1,u_2$ and 
any ball $B\subseteq\Omega$ we have the bound on the H\"older coefficient
\begin{equation}\label{eq:leibhol} [u_1u_2]_{\alpha,B}\leq (\sup_{B}|u_1|)[u_2]_{\alpha,B}+(\sup_B|u_2|)[u_1]_{\alpha,B}. 
\end{equation}
Furthermore, we shall systematically exploit the comparison inequalities \eqref{eq:equiv} and \eqref{eq:equiv2}.
Now given $x\in\Omega$ we let $B=B_{d(x)/2}(x)$, and we estimate the
first type of term
\[[E_1(h_1)(\partial^2 h_1)-E_{1}(h_2)\partial^2 h_2]_{\alpha,B}\leq C\sup_B (|h_1|+|h_2|)[\partial^2 (h_1
-h_2)]_{\alpha,B}+C\sup_B(|\partial^2 h_1|+|\partial^2 h_2|)[h_1-h_2]_{\alpha,B}
\]
\[+C\sup_B|\partial^{2}(h_1-h_2)|([h_1]_{\alpha, B}+[h_2]_{\alpha, B})+C\sup_B(|h_1-h_2|)([\partial^{2}h_1]_{\alpha, B}+[\partial^{2}h_2]_{\alpha, B}).
\]
Using the bounds we have on $h_1$ and $h_2$ we then have
\begin{align*} &[E_1(h_1)(\partial^2 h_1)-E_1(h_2)\partial^2 h_2]_{\alpha,B}\leq Cr_0r(x)^{-p}\phi(x)^{N-n/2-2}(
[\partial^2 (h_1-h_2)]_{\alpha,B}+r(x)^{-2}\phi(x)^{-2}[h_1-h_2]_{\alpha,B}\\
&+r(x)^{-\alpha}\phi^{-\alpha}\sup_B|\partial^{2}(h_1-h_2)|+r(x)^{-2-\alpha}\phi(x)^{-2-\alpha}\sup_B |h_1-h_2|).
\end{align*}
Multiplying by $r(x)^{p+2+\alpha}\phi(x)^{-N+n/2+4+\alpha}$ we have
\begin{align*} &r(x)^{p+2+\alpha}\phi(x)^{-N+n/2+4+\alpha}[E_1(h_1)(\partial^2 h_1)-E_1(h_2)\partial^2 h_2]_{\alpha,B}\leq Cr_0r(x)^{2+\alpha}\phi(x)^{2+\alpha}[\partial^2 (h_1-h_2)]_{\alpha,B}\\
&+Cr_{0}r(x)^\alpha\phi(x)^\alpha[h_1-h_2]_{\alpha,B}+Cr_{0}r(x)^{2}\phi(x)^{2}\sup_B|\partial^{2}(h_1-h_2)|+Cr_{0}\sup_B |h_1-h_2|\\
&\leq Cr_{0}r(x)^{-p}\phi(x)^{N-n/2-2}\left\|h_1-h_2\right\|^{(-p,N-n/2-2)}_{2,\alpha} \leq Cr_0\|(h_1,\omega_1)-(h_2,\omega_2)\|_2
\end{align*}
given our choice of $N$ (recall that we required $N>n+8$).

To handle the second term we estimate 
\begin{align*} & [(E_2(h_1)(h_1)-E_2(h_2)(h_2))(\Gamma^2+\partial\Gamma+\pi^{2})]_{\alpha,B}\leq  Cr(x)^{-p-2}\sup_B|h_1-h_2|([h_1]_{\alpha,B}+[h_2]_{\alpha,B})\\
& +Cr(x)^{-p-2}\sup_B(|h_1|+|h_2|)\{[h_1-h_2]_{\alpha,B}+r(x)^{-\alpha}\phi(x)^{-\alpha}\sup_B |h_1-h_2|\}
\end{align*}
where we have made use of the fact that for $p>0$ one has $2(p+1)>p+2$.
Using the bounds on $h_1$ and $h_2$ and multiplying by $r(x)^{p+2+\alpha}\phi(x)^{-N+n/2+4+\alpha}$ we get
\begin{align*} &r(x)^{p+2+\alpha}\phi(x)^{-N+n/2+4+\alpha}[(E_2(h_1)(h_1)-E_2(h_2)(h_2))(\Gamma^2+\partial\Gamma+\pi^{2})]_{\alpha,B}
\leq Cr_0r(x)^{-p+\alpha}\phi(x)^{2+\alpha}[h_1-h_2]_{\alpha,B}\\
&+Cr_{0}r(x)^{-p}\phi(x)^2\sup_B|h_1-h_2|\leq Cr_{0}r^{-2p}\phi^{N-n/2}\|h_1-h_2\|^{(p,-N+n/2+2)}_{2,\alpha}\leq Cr_0\|(h_1,\omega_1)-(h_2,\omega_2)\|_2.
\end{align*}
Similarly,
\begin{align*} &    [E_3(h)(\partial h)(\Gamma)]_{\alpha,B}\leq Cr(x)^{-2p-1}\phi(x)^{N-n/2-3}([\partial(h_1-h_2)]_{\alpha,B}+r(x)^{-\alpha}\phi(x)^{-\alpha} \sup_{B}|\partial(h_1-h_2)|   \\
& +r(x)^{-1}\phi(x)^{-1}[h_1-h_2]_{\alpha,B}+r(x)^{-1-\alpha}\phi(x)^{-1-\alpha}\sup_B|h_1-h_2|)
\end{align*}
hence
\begin{align*} &  r(x)^{p+2+\alpha}\phi(x)^{-N+n/2+4+\alpha} [E_3(h)(\partial h)(\Gamma)]_{\alpha,B}\leq Cr_0r(x)^{-p+1+\alpha}\phi(x)^{2+\alpha}([\partial(h_1-h_2)]_{\alpha,B} \\
& +r(x)^{-\alpha}\phi(x)^{-\alpha} \sup_{B}|\partial(h_1-h_2)|   +r(x)^{-1}\phi(x)^{-1}[h_1-h_2]_{\alpha,B}+r(x)^{-1-\alpha}\phi(x)^{-1-\alpha}\sup_B|h_1-h_2|)      \\
&\leq Cr_0 \|(h_1,\omega_1)-(h_2,\omega_2)\|_2. 
\end{align*}

To estimate the fourth type of term we follow the same pattern
\begin{align*} &[E_4(h_1)(\partial h_1)(\partial h_1)-E_4(h_2)(\partial h_2)(\partial h_2)]_{\alpha,B}\leq Cr(x)^{-\alpha}\phi(x)^{-\alpha}\sup_B(|\partial h_1|+|\partial h_2|)\sup_{B}|\partial (h_1-h_2)| \\
&+C\left(\sup_B(|\partial h_1|+|\partial h_2|)[\partial (h_1-h_2)]_{\alpha,B}+\sup_B \left|\partial(h_{1}-h_{2})\right|\left([\partial h_{1}]_{\alpha, B}+[\partial h_{2}]_{\alpha, B}\right)\right) \\
&+C\left(\sup_B(|\partial h_1|^2+|\partial h_2|^2)[h_1-h_2]_{\alpha,B}+\sup_{B}|h_{1}-h_{2}|([(\partial h_{1})^{2}]_{\alpha,B}+[(\partial h_{2})^{2}]_{\alpha,B})\right)
\end{align*}
hence
\begin{align*} &r(x)^{p+2+\alpha}\phi(x)^{-N+n/2+4+\alpha}[E_4(h_1)(\partial h_1)(\partial h_1)-E_4(h_2)(\partial h_2)(\partial h_2)]_{\alpha,B}\\ 
&\leq Cr_{0}r(x)\phi(x) (\sup_{B}|\partial(h_1-h_{2})|+r(x)^{\alpha}\phi(x)^{\alpha}[\partial(h_1-h_2)]_{\alpha,B}) \\
& +Cr_{0}^{2}r(x)^{-p}\phi(x)^{N-n/2-2}(\sup_{B}|h_{1}-h_{2}|+ r(x)^{\alpha}\phi(x)^{\alpha}[h_1-h_2]_{\alpha,B})\\
&\leq Cr_{0}r(x)^{-p}\phi(x)^{N-n/2-2}\left\|h_{1}-h_{2}\right\|^{(-p,N-n/2-2)}_{2,\alpha}\leq Cr_{0}\left\|(h_{1},\omega_2)-(h_{2},\omega_2)\right\|_{2}.
\end{align*}
Lastly, we can take of the error terms involving $E_{5}$ and $E_{6}$:
\begin{align*}
&[E_5(h_1)(\omega_1)(\omega_1)-E_5(h_2)(\omega_2)(\omega_2)]_{\alpha,B}\leq r(x)^{-\alpha}\phi(x)^{-\alpha}\sup_B (|\omega_{1}|+|\omega_{2}|)\sup_B(|\omega_1-\omega_2|)\\
& +C\left(\sup_B |\omega_1-\omega_2|\left([\omega_1]_{\alpha,B}+[\omega_2]_{\alpha,B}\right)+\sup_B\left(\left|\omega_1\right|+\left|\omega_2\right|\right)[\omega_1-\omega_2]_{\alpha, B}\right)\\
& +C\left(\sup_B (\left|\omega_1\right|^{2}+\left|\omega_2\right|^{2})[h_1-h_2]_{\alpha, B}+\sup_{B}|h_{1}-h_{2}|([\omega_{1}^{2}]_{\alpha,B}+[\omega_{2}^{2}]_{\alpha,B})\right) 
\end{align*}
and thus
\begin{align*} &r(x)^{p+2+\alpha}\phi(x)^{-N+n/2+4+\alpha}[E_5(h_1)(\omega_1)(\omega_1)-E_5(h_2)(\omega_2)(\omega_2)]_{\alpha,B}\\ 
&\leq Cr_{0}r(x)\phi(x) (\sup_{B}|\omega_1-\omega_2|+r(x)^{\alpha}\phi(x)^{\alpha}[\omega_1-\omega_2]_{\alpha,B})\\
&+Cr_{0}^{2}r(x)^{-p}\phi(x)^{N-n/2}(|h_{1}-h_{2}|+r(x)^{\alpha}\phi(x)^{\alpha}[h_1-h_2]_{\alpha,B}\\
&\leq Cr_{0}r(x)^{-p}\phi(x)^{N-n/2-2}(\left\|h_{1}-h_{2}\right\|^{(-p,N-n/2-2)}_{2,\alpha}+\left\|\omega_{1}-\omega_{2}\right\|^{(-p-1,N-n/2-2)}_{2,\alpha})\\
&\leq Cr_{0}\left\|(h_{1},\omega_1)-(h_{2},\omega_2)\right\|_{2}
\end{align*}
on the one hand, and
\begin{align*}
&[E_6(h_1)(h_1)(\omega_1)(\pi)-E_6(h_2)(h_2)(\omega_2)(\pi)]_{\alpha,B}\\
&\leq Cr(x)^{-p-1-\alpha}\sup_{B}\left(|h_1|+|h_2|\right)(\sup_{B}|\omega_1-\omega_2|+r(x)^{\alpha}\phi(x)^{\alpha}[\omega_1-\omega_2]_{\alpha, B})\\
&+Cr(x)^{-p-1-\alpha}\sup_{B}\left(|\omega_1|+|\omega_2|\right)(\sup_B |h_1-h_2|+r(x)^{\alpha}\phi(x)^{\alpha}[h_1-h_2]_{\alpha, B})\\
&+Cr(x)^{-p-1}\sup_B |h_1-h_2|([\omega_1]_{\alpha, B}+[\omega_2]_{\alpha, B})+Cr(x)^{-p-1}\sup_B |\omega_1-\omega_2|([h_1]_{\alpha, B}+[h_2]_{\alpha, B})
\end{align*}
hence
\begin{align*}
&r(x)^{p+2+\alpha}\phi(x)^{-N+n/2+4+\alpha}[E_6(h_1)(h_1)(\omega_1)(\pi)-E_6(h_2)(h_2)(\omega_2)(\pi)]_{\alpha,B}\\
&\leq Cr_{0}r(x)^{-p+1+\alpha}\phi(x)^{2}(\sup_B |\omega_1-\omega_2|+r(x)^{\alpha}\phi(x)^{\alpha}[\omega_1-\omega_2]_{\alpha,B})\\
&+Cr_{0}r(x)^{-p+\alpha}\phi(x)^{2}(\sup_B |h_1-h_2|+r(x)^{\alpha}\phi(x)^{\alpha}[h_1-h_2]_{\alpha,B})\\
&\leq Cr_{0}r(x)^{-2p+\alpha}\phi(x)^{N-n/2}(\left\|h_1-h_2\right\|^{(-p,N-n/2-2)}_{2,\alpha}+\left\|\omega_1-\omega_2\right\|^{(-p-1,N-n/2-2)}_{2,\alpha})\\
&\leq Cr_0 \left\|(h_1,\omega_1)-(h_2,\omega_2)\right\|
\end{align*}
on the other.

In order to get $\mathcal{C}^{1,\alpha}$ weighted estimates of $Q^{(2)}_{(g,\pi)}$ we first need to differentiate each of the espressions $E_7,E_8,E_9,E_{10}$  we got above. Via elementary, though somewhat lenghty computations, one can use such expressions to write $\partial Q^{(2)}_{(g,\pi)}$ as a sum of terms belonging to one of finitely many types that can be explicitly listed and we omit the tedious details.
Arguing as we did before and making use of the $\mathcal{C}^{2,\alpha}$ estimates on $h_{1}, h_{2}$ as well as $\omega_{1}, \omega_{2}$ one proves the pointwise estimates:
\[|Q^{(2)}_{(g,\pi)}[h_{1},\omega_{1}]-Q^{(2)}_{(g,\pi)}[h_{2},\omega_{2}]|\leq Cr_{0} r^{-2p-2}\phi^{2(N-n/2-4)}\left\|(h_{1}-h_{2},\omega_{1}-\omega_{2})\right\|_{2}
\]
and
\[ |\partial Q^{(2)}_{(g,\pi)}[h_{1},\omega_{1}]-\partial Q^{(2)}_{(g,\pi)}[h_{2},\omega_{2}]|\leq Cr_{0} r^{-2p-3}\phi^{2(N-n/2-5)}\left\|(h_{1}-h_{2},\omega_{1}-\omega_{2})\right\|_{2}
\]
which patently imply
\[ r^{p+2}\phi^{-N+n/2+3}|Q^{(2)}_{(g,\pi)}[h_{1},\omega_{1}]-Q^{(2)}_{(g,\pi)}[h_{2},\omega_{2}]|+r^{p+3}\phi^{-N+n/2+4}|\partial Q^{(2)}_{(g,\pi)}[h_{1},\omega_{1}]-\partial Q^{(2)}_{(g,\pi)}[h_{2},\omega_{2}]|\leq Cr_{0}
\]
as long as $N>n/2-6$.
Similarly, making repeated use of the inequality \eqref{eq:leibhol} one shows that for any point $x\in\Omega$
\[ \left[Q^{(2)}_{(g,\pi)}[h_{1},\omega_{1}]-Q^{(2)}_{(g,\pi)}[h_{2},\omega_{2}]\right]_{\alpha,B_{d(x)/2}}\leq r^{-2p-2-\alpha}\phi^{2(N-n/2-4-\alpha)}\left\|(h_{1}-h_{2},\omega_{1}-\omega_{2})\right\|_{2}
\]
as well as
\[\left[\partial Q^{(2)}_{(g,\pi)}[h_{1},\omega_{1}]-\partial Q^{(2)}_{(g,\pi)}[h_{2},\omega_{2}]\right]_{\alpha, B_{d(x)/2}}\leq r^{-2p-3-\alpha}\phi^{2(N-n/2-5-\alpha)}\left\|(h_{1}-h_{2},\omega_{1}-\omega_{2})\right\|_{2}
\]
so that in the end
\[ \left\|Q^{(2)}_{(g,\pi)}[h_{1},\omega_{1}]-Q^{(2)}_{(g,\pi)}[h_{2},\omega_{2}] \right\|^{(p+2,-N+n/2+3)}_{1,\alpha}\leq Cr_{0}
\]
for some constant $C$ only depending on the ambient dimension.

Combining both our integral and pointwise estimates, we have shown that
\[\left\|\left(Q^{(1)}_{(g,\pi)}[h_{1},\omega_{1}]- Q^{(1)}_{(g,\pi)}[h_{2},\omega_{2}], Q^{(2)}_{(g,\pi)}[h_{1},\omega_{1}]- Q^{(2)}_{(g,\pi)}[h_{2},\omega_{2}]\right)\right\|_{1}\leq Cr_{0}\left\|(h_{1},\omega_{1})-(h_{2},\omega_{2})\right\|_{2}
\]
if $\left\|(h_{1},\omega_{1})\right\|_{1}+\left\|(h_{2},\omega_{2})\right\|_{1}\leq r_{0}$.
Thus, given $\lambda>0$ as small as we need, we simply pick $r_{0}=\lambda/4C$ and this completes the proof. \end{proof}

\subsection{Continuity of the ADM energy-momentum} \label{subs:sobcont}

We shall be concerned here with the relation between the ADM energy-momentum of our initial data $(M,\check{g}, \check{k})$ and those of its localization $(M,\hat{g},\hat{k})$ as in the statement of Theorem \ref{thm:main}. In order to simplify the exposition, we will assume that the manifold $M$ has only one end and yet this restriction is completely unnecessary since our construction happens \textsl{one end at a time}. We know that, by our very gluing scheme, $\hat{g}=\check{g}$ and $\hat{k}=\check{k}$ on the whole region $\Omega_{I}$ so that, in particular, this is true in the interior of a (Euclidean) ball of radius $|a|\sin\theta_{1}$ (resp. $|a|/2$) for $\theta_1\in(0,\pi/2)$ (resp. for $\theta_1\in[\pi/2,\pi)$). (We are assuming the ball in question to be defined by the usual equation, given in asymptotically flat coordinates $\left\{x\right\}$ along the unique end of $M$). As a result, it follows from our construction that for any $\overline{p}\in\left(\frac{n-2}{2},p\right)$
\[ \hat{g}\stackrel{\mathcal{H}_{2,-\overline{p}}}{\rightarrow}\check{g}, \ \ \hat{\pi}\stackrel{\mathcal{H}_{1,-1-\overline{p}}}{\rightarrow}\check{\pi} \ \ \ \textrm{as} \ \ |a|\to\infty.
\]
Notice that the previous statement is true both when one considers the volume measure associated to $\check{g}$ and when one considers the volume measure associated to $\hat{g}$ or in fact any other measure on $M$ which is equivalent to $\mathscr{L}^{n}$ outside a large compact set.
Our target now is to prove that the ADM energy-momentum is continuous with respect to sequential convergence in those topologies. This is probably known to the experts, but not so easily found in the literature, so we give here a detailed argument to fill this gap.

\begin{proposition}\label{pro:ADMcont}
Let $(M,g_{\infty},k_{\infty})$ be an asymptotically flat initial data set with $g_{\infty}\in\mathcal{H}_{2,-p}$ and $k_{\infty}\in\mathcal{H}_{1,-p-1}$ for some $p>\frac{n-2}{2}$. Consider a sequence of data $(M,g_{l},k_{l})$ such that $g_{l}\to g_{\infty}$ in $\mathcal{H}_{2,-p}$ and $k_{l}\to k_{\infty}$ in $\mathcal{H}_{1,-p-1}$ as $l$ tends to infinity. Then 
\[ \lim_{a\to\infty}\mathcal{E}^{(l)}=\mathcal{E}^{(\infty)}, \ \ \lim_{a\to\infty}\mathcal{P}^{(l)}=\mathcal{P}^{(\infty)}.
\] 
Here $\mathcal{E}^{(l)}, \mathcal{P}^{(l)}$ (resp. $\mathcal{E}^{(\infty)}, \mathcal{P}^{(\infty)}$) denote the ADM energy-momentum of $(M,g_{l},k_{l})$ (resp. $(M,g_{\infty},k_{\infty})$).
\end{proposition}

To avoid dangerous ambiguities, let us remark that (by our definition of initial data sets, see Subsection \ref{subs:ids}) we are assuming both $(M,g_{\infty},k_{\infty})$ and each $(M,g_{l},k_{l})$ to solve the Einstein constraint equations. For the sake of simplicity, we will deal here with the \textsl{vacuum} case, even though an identical proof can be given for the general constraints under the usual integrability assumption on $\mu$ and $J$.

\begin{proof}
Let us first consider the energy functional. Given a large radius $r'$ we start by claiming that for any asymptotically flat metric $g$ as in the statement of our proposition
\[\left|\mathcal{E}-\frac{1}{2(n-1)\omega_{n-1}}\int_{|x|=r'}\sum_{i,j=1}^{n}(g_{ij,i}-g_{ii,j})\nu_{0}^{j}\,d\mathscr{H}^{n-1}\right|\leq C(r')^{2n-4-4p}
\]
(for some constant $C$ which is uniform in a given $\mathcal{H}_{2,-p}\times\mathcal{H}_{1,-p-1}$ neighbourhood of $(g,k)$) and it is clear that once this is justified, it is enough to prove the continuity (in our topology) of the approximating hypersurface integral at \textsl{fixed} radius $r'$.
To obtain the claim, let us pick a second, larger radius $r''>r'$ and notice that 
\[ \int_{|x|=r''}\sum_{i,j=1}^{n}(g_{ij,i}-g_{ii,j})\nu_{0}^{j}\,d\mathscr{H}^{n-1}-\int_{|x|=r'}\sum_{i,j=1}^{n}(g_{ij,i}-g_{ii,j})\nu_{0}^{j}\,d\mathscr{H}^{n-1}=\int_{B_{r''}\setminus B_{r'}}(R_g+\mathcal{R})\,d\mathscr{L}^{n}
\]
(with a remainder term satisfying $|\mathcal{R}|\leq C(|g-\delta||\partial^{2}g|+|\partial g|^{2})$) by the divergence theorem, hence by the first constraint equation (the Hamiltonian constraint) that integral is upper bounded by (modulo a dimensional constant)
\[ \int_{M\setminus B_{r'}}(|g-\delta||\partial^{2}g|+|\partial g|^{2}+|\pi|^{2})\,d\mathscr{L}^{n}.
\]
The terms $|g-\delta||\partial^{2}g|, \ |\partial g|^{2}$ and $|\pi|^{2}$ have, under our assumptions, roughly the same rate of decay (of order $|x|^{-2(p+1)}$, in integral sense) so we will just consider the latter, the treatment for the former two being identical.
We can write:
\[\int_{M\setminus B_{r'}}|\pi|^{2}\,d\mathscr{L}^{n}\leq \left(\int_{M\setminus B_{r'}}(|\pi|r^{1+p})^{2}r^{-n}\,d\mathscr{L}^{n}\right)^{1/2}\left(\int_{M\setminus B_{r'}}(|\pi|r^{n-1-p})^{2}r^{-n}\,d\mathscr{L}^{n}\right)^{1/2}
\]
and in turn
\[ \int_{M\setminus B_{r'}}(|\pi|r^{n-1-p})^{2}r^{-n}\,d\mathscr{L}^{n}\leq \sup_{r>r'} r^{2n-4-4p} \times \left(\int_{M\setminus B_{r'}}|\pi|^{2}r^{-n+2p+2}\,d\mathscr{L}^{n}\right)
\] which completes the proof of our claim.
Therefore, for any given $\varepsilon>0$ let us fix a reference radius $r'$ so that for each metric $g_{i}$ as well as for $g_{\infty}$ the ADM energy is appromixated by the integral over the hypersphere of radius $r'$ (considering the measure $\mathscr{H}^{n-1}$) with an error less that $\varepsilon/2$.  
Hence, applying the divergence theorem once again
\begin{align*} & \int_{|x|=r'}\sum_{i,j=1}^{n}(g^{(\infty)}_{ij,i}-g^{(\infty)}_{ii,j})\nu_{0}^{j}\,d\mathscr{H}^{n-1}-\int_{|x|=r'}\sum_{i,j=1}^{n}(g^{(l)}_{ij,i}-g^{(l)}_{ii,j})\nu_{0}^{j}\,d\mathscr{H}^{n-1}\\
& =\int_{|x|<r'}(g^{(\infty)}_{ij,ij}-g^{(\infty)}_{ii,jj})\,d\mathscr{L}^{n}-\int_{|x|<r'}(g^{(l)}_{ij,ij}-g^{(l)}_{ii,jj})\,d\mathscr{L}^{n} \\
\end{align*}
thus
\begin{align*}
& \left|\int_{|x|=r'}\sum_{i,j=1}^{n}(g^{(\infty)}_{ij,i}-g^{(\infty)}_{ii,j})\nu_{0}^{j}\,d\mathscr{H}^{n-1}-\int_{|x|=r'}\sum_{i,j=1}^{n}(g^{(l)}_{ij,i}-g^{(l)}_{ii,j})\nu_{0}^{j}\,d\mathscr{H}^{n-1}\right| \\
& \leq C \left(\int_{M}\left|g^{(\infty)}_{ij,ij}-g^{(l)}_{ij,ij}\right|^{2}r^{-n+2p+4}\,d\mathscr{L}^{n}\right)^{1/2}\left(\int_{|x|<r'}r^{n-2p-4}\,d\mathscr{L}^{n}\right)^{1/2} \\
&+C\left(\int_{M}\left|g^{(\infty)}_{ii,jj}-g^{(l)}_{ii,jj}\right|^{2}r^{-n+2p+4}\,d\mathscr{L}^{n}\right)^{1/2}\left(\int_{|x|<r'}r^{n-2p-4}\,d\mathscr{L}^{n}\right)^{1/2}
\end{align*}
By our convergence assumption, we can now pick an index $l_{0}$ so that for $l>l_{0}$ each of those two summands is less than $\varepsilon/2$ so that in the end $|\mathcal{E}_{\infty}-\mathcal{E}_{l}|<\varepsilon$. Therefore, since $\varepsilon$ can be chosen arbitrarily small we conclude that the energy functional is continuous in these weighted Sobolev spaces. The proof of the continuity of the ADM linear momentum follows along the same lines, so we will only sketch it. First of all, in order to show that
\[\left|\mathcal{P}_i-\frac{1}{(n-1)\omega_{n-1}}\int_{|x|=r'}\sum_{j=1}^{n}\pi_{ij}\nu_{0}^{j}\,d\mathscr{H}^{n-1}\right|\leq C(r')^{2n-4-4p}
\]
one applies the divergence theorem observing that
\[
 \sum_j\pi_{ij,i}=\sum_j(\pi_{ij;j}-\pi_{ij,j})\leq C |\partial g||\pi|
 \]
by virtue of the second constraint equation (the vector constraint) and the fact that, in local coordinates 
\[ \pi_{ij;i}=\pi_{ij,j}-\Gamma^{k}_{jj}\pi_{ik}-\Gamma^{k}_{ij}\pi_{jk}.
\]

Since $|\partial g||\pi|\leq |\partial g|^{2}+|\pi|^{2}$ the claim follows by the argument we have already presented because of course we only need to deal with terms of order zero or one. At that stage, the comparison happens at the level of a fixed hypersphere and the convergence assumption can be used in a straightforward fashion. This concludes the proof. 
\end{proof}

As a result, we can immediately deduce the surprising fact that our localization construction gives an arbitrarily good approximation of the ADM energy-momentum of the given initial data, no matter how small the angles $\theta_{1}$ and $\theta_{2}$ may potentially be.

\begin{corollary}
Let $(M,\check{g},\check{k})$, $(M,\hat{g},\hat{k})$ be as in the statement of Theorem \ref{thm:main}. Then
\[ \lim_{a\to\infty} \hat{\mathcal{E}}=\check{\mathcal{E}}, \ \ \lim_{a\to\infty}\hat{\mathcal{P}}=\check{\mathcal{P}}.
\]
Here $\hat{\mathcal{E}}, \hat{\mathcal{P}}$ (resp. $\check{\mathcal{E}}, \check{\mathcal{P}}$) denote the ADM energy-momentum of $(M,\hat{g},\hat{k})$ (resp. $(M,\check{g},\check{k})$).
\end{corollary}

\section{A new class of $N-$body solutions}\label{sec:Nbody}

As anticipated in the Introduction, we would like to devote this section to a discussion about the applicability of our gluing methods to the construction of a new class of $N$-body initial data sets for the Einstein constraint equations. In the context of Newtonian gravity, a set of initial data for the evolution of $N$ massive bodies can be obtained by solving a single Poisson equation in the complement of a finite number of compact domains (say balls) in $\R^{3}$, at least if the interior structure and dynamics of the bodies in question is neglected. The nonlinearity of the Einstein constraint equations makes such a task a lot harder and in fact it was only in the last decade that sufficiently general results in this direction, without restricting to rather symmetric configurations \cite{CD03} or exploiting the presence of multiple ends \cite{CIP05}, have been obtained \cite{CCI09, CCI11}.

\

Let us suppose that a finite collection of $N$ asymptotically flat data sets $(M_{1}, g_{1}, k_{1}), \ldots, (M_{N}, g_{N}, k_{N})$ is assigned and let $U_{i}$ denote a compact, regular subdomain of $M_{i}$ for each value of the index $i$. Moreover, let $x_{1},\ldots, x_{N}\in \R^{n}$ be $N$ vectors which are supposed to prescribe the location of the regions $U_{1}, \ldots, U_{k}$ with respect to a fiducial flat background. Then, the problem we want to address is the construction of a triple $(M,g,k)$ which:
\begin{itemize}
\item{is a solution of the vacuum Einstein constraint equations $\Phi(g,k)=0$;}
\item{contains $N$ regions that are isometric to the given bodies;}
\item{has the centers of such bodies in a configuration which is a scaled version of the chosen configuration.} 
\end{itemize}
\

Making use of Theorem \ref{thm:main}, we can glue together any finite number of conical regions for given data $(M_{i}, g_{i}, k_{i})$. In order to state this result, let us introduce some notation. Given $n\geq 3$, $a\in \R^{n}$, $\theta\in(0,\pi)$ and $\varepsilon>0$ very small we let $\Gamma_{\theta}\left(a\right)$ (resp. $\Omega_{\theta,\varepsilon}\left(a\right)$) be the regularized conical region obtained from $C_{\theta}(a)$ (resp. $C_{\theta}(a)\setminus C_{\left(1-\varepsilon\right)\theta}\left(a\right)$) by removing the singularity at the tip and smoothly gluing $B_{1/2}\left(a\right)$ as explained in the Subsection \ref{subs:regcon}. In addition, we set 
\[\hat{\Gamma}_{\theta,\varepsilon}\left(a\right)=\Gamma_{\theta}\left(a\right)-(1+\varepsilon)a, \ \
 \hat{\Omega}_{\theta,\varepsilon}\left(a\right)=\Omega_{\theta, \varepsilon}\left(a\right)-(1+\varepsilon)a
 \]    
that are translated copies of the smoothened regions. We remark that for any given $\varepsilon>0$ if $a, a'$ are large enough we have that $\left[C_{\theta}(a)-a\right]\cap \left[C_{\theta'}(a')-a'\right]=\emptyset$ implies $\hat{\Omega}_{\theta,\varepsilon}\left(a\right)\cap \hat{\Omega}_{\theta',\varepsilon}\left(a'\right)=\emptyset$ as well.

Finally, given two vectors $v_{1}, v_{2}\in \mathbb{S}^{n-1}$ let $\varphi(v_{1}, v_{2})\in\left[0,\pi\right]$ be the angle between them. This statement is a remarkable consequence of our gluing Theorem \ref{thm:main}.

\begin{theorem}\label{thm:Nbody}
Given $\sigma>0, 0<\varepsilon<1/2$ and a collection of initial data $(M_{1}, g_{1}, k_{1}), \ldots, (M_{N}, g_{N}, k_{N})$ (as in Subsection \ref{subs:ids}) satisfying the vacuum Einstein constraint equations one can find $\Lambda>0$ such that: assigned vectors $a_{1},\ldots, a_{N}\in\R^{n}$ and angles $\theta_{1}, \ldots,\theta_{N}$ satisfying $\left|a_{i}\right|>\Lambda$ and $\varphi(a_{i}, a_{j})>(1-\varepsilon)^{-1}(\theta_{i}+\theta_{j})$ (for any choice of indices $i, j$) then there exists an asymptotically Euclidean triple $(M,g, k)$ satisfying the vacuum Einstein constraint equations, where $M$ has exactly one end, such that $g$ (resp. $k$) on $\hat{\Gamma}_{\theta_{i},\varepsilon}\left(a_{i}\right)$ coincides with $g_{i}$ (resp. $k=k_{i}$) on $\Gamma_{\theta_{i},\varepsilon}\left(a_{i}\right)\subset M_i$ and 
\[ \left|\mathcal{E}-\sum_{i=1}^{N}\mathcal{E}^{(i)}\right|\leq\sigma, \ \ \left|\mathcal{P}-\sum_{i=1}^{N}\mathcal{P}^{(i)}\right|\leq\sigma.
\]
\end{theorem}

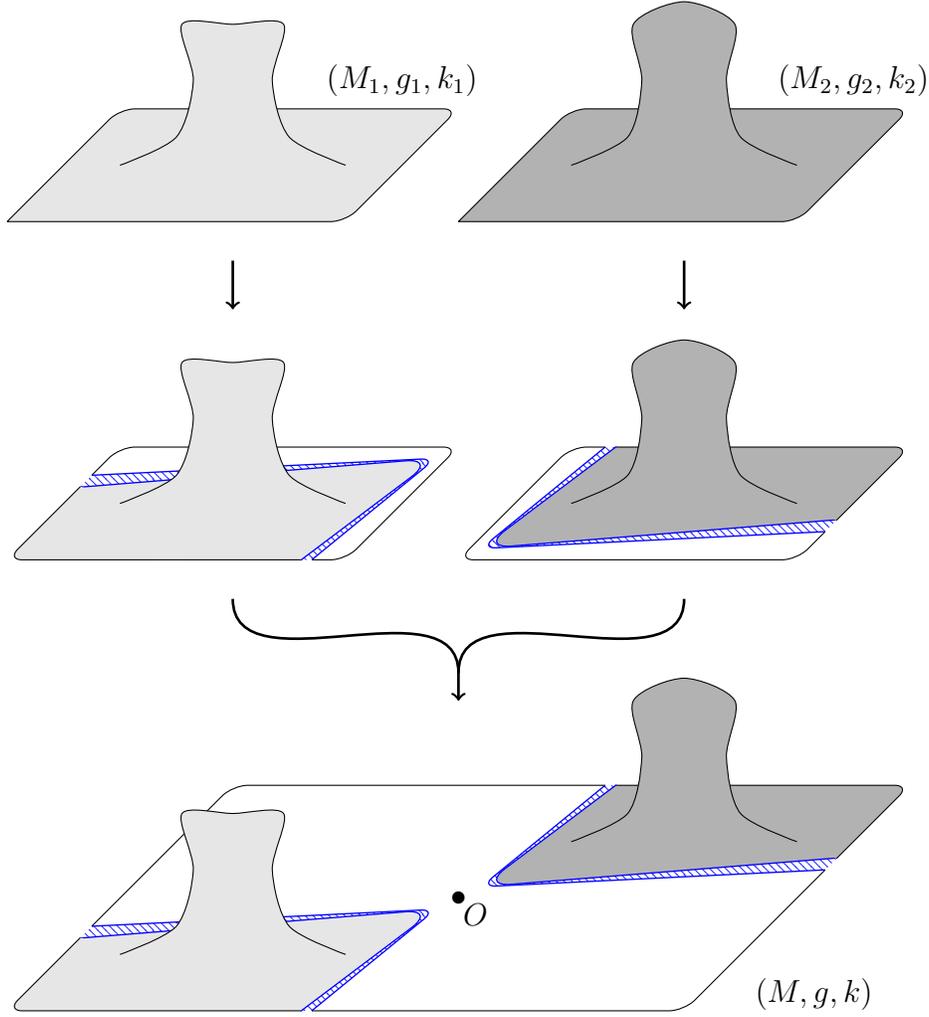
\begin{figure}[ht!]

	\begin{tikzpicture}
	\begin{scope}[scale=0.75]
	
	\begin{scope}[shift={(0, 6)}]
	\begin{scope}[shift={(-4, 0)}]
	\fill [black!10, draw=black, rounded corners=2mm] plot coordinates {(-2, -1) (0, 1) (6, 1) (4, -1) (-2, -1)};
	\fill [black!10, draw=black] plot [smooth] coordinates {(0, 0) (1, 0.5) (1.3, 1.5) (1.1, 2.5) (2, 2.5) (2.9, 2.5) (2.7, 1.5) (3, 0.5) (4, 0)};
	\node at (5, 1.5) {$(M_1, g_1, k_1)$};
	\end{scope}
	
	\begin{scope}[shift={(4, 0)}]
	\fill [black!30, draw=black, rounded corners=2mm] plot coordinates {(-2, -1) (0, 1) (6, 1) (4, -1) (-2, -1)};
	\fill [black!30, draw=black] plot [smooth] coordinates {(0, 0) (1, 0.5) (1.25, 1.5) (1.1, 2.5) (2, 2.9) (2.9, 2.5) (2.75, 1.5) (3, 0.5) (4, 0)};
	\node at (5, 1.5) {$(M_2, g_2, k_2)$};
	\end{scope}
	
	\node at (-2, -1.5) (M1) {};
	\node at (6, -1.5) (M2) {};
	\node at (-2, -2.75) (M1tempup) {};
	\node at (6, -2.75) (M2tempup) {};
	\path[every node/.style={font=\sffamily\small}, ->, line width=1pt] (M1) edge [out=270, in=90] (M1tempup);
	\path[every node/.style={font=\sffamily\small}, ->, line width=1pt] (M2) edge [out=270, in=90] (M2tempup);
	
	\end{scope}
	
	\begin{scope}[shift={(-4, 0)}]
	\fill [pattern color=blue!80, pattern=north west lines, draw=blue, line width=0.5pt, rounded corners=2mm] plot coordinates {(-0.5, 0.5) (5.6, 0.8) (3.4, -1)};
	\fill [pattern color=blue!80, pattern=north west lines, draw=none, rounded corners=2mm] plot coordinates {(3.4, -1) (-2, -1) (-0.5, 0.5)};
	\fill [black!10, draw=blue, line width=0.5pt, rounded corners=4mm] plot coordinates {(-0.7, 0.29) (5.6, 0.8) (3.2, -1.01)};
	\fill [black!10, draw=black, rounded corners=2mm] plot coordinates {(3.2, -1) (-2, -1) (-0.7, 0.3)};
	\draw [black, rounded corners=2mm] plot coordinates {(-0.5, 0.5) (0, 1) (6, 1) (4, -1) (3.4, -1)};
	\fill [black!10, draw=black] plot [smooth] coordinates {(0, 0) (1, 0.5) (1.3, 1.5) (1.1, 2.5) (2, 2.5) (2.9, 2.5) (2.7, 1.5) (3, 0.5) (4, 0)};
	\end{scope}
	
	\begin{scope}[shift={(4, 0)}]
	\fill [pattern color=blue!80, pattern=north west lines, draw=blue, line width=0.5pt, rounded corners=2mm] plot coordinates {(0.6, 1) (-1.6, -0.8) (4.5, -0.5)};
	\fill [pattern color=blue!80, pattern=north west lines, draw=none, rounded corners=2mm] plot coordinates {(4.5, -0.5) (6, 1) (0.6, 1)};
	\fill [black!30, draw=blue, line width=0.5pt, rounded corners=4mm] plot coordinates {(0.8, 1.01) (-1.6, -0.8) (4.7, -0.29)};
	\fill [black!30, draw=black, rounded corners=2mm] plot coordinates {(4.7, -0.3) (6, 1) (0.8, 1)};
	\draw [black, rounded corners=2mm] plot coordinates {(0.6, 1) (0, 1) (-2, -1) (4, -1) (4.5, -0.5)};
	\fill [black!30, draw=black] plot [smooth] coordinates {(0, 0) (1, 0.5) (1.25, 1.5) (1.1, 2.5) (2, 2.9) (2.9, 2.5) (2.75, 1.5) (3, 0.5) (4, 0)};
	\end{scope}
	
	\node at (-2, -1.5) (M1temp) {};
	\node at (6, -1.5) (M2temp) {};
	\node at (2, -3.2) (M) {};
	\path[every node/.style={font=\sffamily\small}, line width=1pt] (M1temp) edge [out=270, in=90] (M);
	\path[every node/.style={font=\sffamily\small}, line width=1pt] (M2temp) edge [out=270, in=90] (M);
	\draw [black, ->, line width=1pt] (2, -3) -- (2, -3.5);
	
	\begin{scope}[shift={(0, -7)}]
	\begin{scope}[shift={(-4, -1)}]
	\fill [pattern color=blue!80, pattern=north west lines, draw=blue, line width=0.5pt, rounded corners=2mm] plot coordinates {(-0.5, 0.5) (5.6, 0.8) (3.4, -1)};
	\fill [pattern color=blue!80, pattern=north west lines, draw=none, rounded corners=2mm] plot coordinates {(3.4, -1) (-2, -1) (-0.5, 0.5)};
	\fill [black!10, draw=blue, line width=0.5pt, rounded corners=4mm] plot coordinates {(-0.7, 0.29) (5.6, 0.8) (3.2, -1.01)};
	\fill [black!10, draw=black, rounded corners=2mm] plot coordinates {(3.2, -1) (-2, -1) (-0.7, 0.3)};
	\end{scope}
	
	\begin{scope}[shift={(4, 1)}]
	\fill [pattern color=blue!80, pattern=north west lines, draw=blue, line width=0.5pt, rounded corners=2mm] plot coordinates {(0.6, 1) (-1.6, -0.8) (4.5, -0.5)};
	\fill [pattern color=blue!80, pattern=north west lines, draw=none, rounded corners=2mm] plot coordinates {(4.5, -0.5) (6, 1) (0.6, 1)};
	\fill [black!30, draw=blue, line width=0.5pt, rounded corners=4mm] plot coordinates {(0.8, 1.01) (-1.6, -0.8) (4.7, -0.29)};
	\fill [black!30, draw=black, rounded corners=2mm] plot coordinates {(4.7, -0.3) (6, 1) (0.8, 1)};
	\node at (4.3, -2.7) {$(M, g, k)$};
	\end{scope}
	
	\node at (2, 0) {$\bullet$};
	\node at (2.3, -0.3) {$O$};
	
	\draw [black, rounded corners=2mm] plot coordinates {(-4.5, -0.5) (-2, 2) (4.6, 2)};
	\draw [black, rounded corners=2mm] plot coordinates {(8.5, 0.5) (6, -2) (-0.6, -2)};
	
	\begin{scope}[shift={(-4,-1)}]
	\fill [black!10, draw=black] plot [smooth] coordinates {(0, 0) (1, 0.5) (1.3, 1.5) (1.1, 2.5) (2, 2.5) (2.9, 2.5) (2.7, 1.5) (3, 0.5) (4, 0)};
	\end{scope}
	
	\begin{scope}[shift={(4,1)}]
	\fill [black!30, draw=black] plot [smooth] coordinates {(0, 0) (1, 0.5) (1.25, 1.5) (1.1, 2.5) (2, 2.9) (2.9, 2.5) (2.75, 1.5) (3, 0.5) (4, 0)};
	\end{scope}
	
	\end{scope}
	
	\end{scope} 
	\end{tikzpicture}

	\caption{Theorem \ref{thm:Nbody} allows to merge an assigned collection of data into an exotic $N-$body solution of the Einstein constraint equations.}
	\label{dgm:gluing.figure}
\end{figure}

\begin{remark}
Let us recall from Subsection \ref{subs:regcon} that our data are all tacitly assumed to have only one end. If instead this is not the case, the same construction goes through anyway, provided each $(M_i,g_i,k_i)$ comes with a preferred, labelled end on which the gluing is performed: in this case the resulting triple will of course have multiple ends and indeed
\[
\# \ \textrm{ends} \ (M,g,k) = 1+\sum_{i=1}^{N}( \# \ \textrm{ends} \ (M_i,g_i,k_i) -1).
\]	
	
\end{remark}	

\begin{proof} First of all, one can find $\Lambda$ large enough that Theorem \ref{thm:main} is applicable to every single triple $(M_i,g_i,k_i)$ and, under the additional constraint $\varepsilon\Lambda\max_{i,j}\left\{\sin\left(\frac{\varphi(a_i,a_j)}{2}\right)\right\}>1$ we can make sure that each couple of domains $\hat{\Gamma}_{\theta_{i}/\left(1-\varepsilon\right), \varepsilon}\left(a_{i}\right)$ will indeed be mutually disjoint whenever $|a_i|>\Lambda$ for any $i$ and $\varphi(a_{i}, a_{j})>(1-\varepsilon)^{-1}(\theta_{i}+\theta_{j})$, as in the statement above. Thus, we employ our gluing theorem exactly $N$ times to construct on $\R^{n}\setminus \cup_{i=1}^{N}\hat{\Gamma}_{\left(1-\varepsilon\right)\theta_{i}, \varepsilon}\left(a_{i}\right)$ a smooth couple $(g,k)$  that agrees with each of our data $(g_{i}, k_{i})$ at the interface $\hat{\Omega}_{\theta_{i}, \varepsilon}\left(a_{i}\right)$ and is exactly Euclidean outside of $\hat{\Gamma}_{\theta_{i}/\left(1-\varepsilon\right), \varepsilon}\left(a_{i}\right)$.

	 As a second step, one can \textsl{fill in} the conical subregions by smoothly identifying the boundary of such manifold with the boundary of the conical domains $\hat{\Gamma}_{\left(1-\varepsilon\right)\theta_{i}, \varepsilon}\left(a_{i}\right)$ in each $M_{i}$. Correspondingly, one can extend the tensors $g$ and $k$ on the whole $M$.

We claim that the resulting triple $(M,g,k)$ satisfies all of our requirements and to that aim it is enough to check the last assertion. As remarked in Subsection \ref{subs:sobcont}, it follows from our general construction that 
\[\mathcal{E}\left(\hat{g}\right)\to \mathcal{E}\left(g\right) \ \ \textrm{as} \ \left|a\right|\to\infty 
\]
and similarly, for the linear momentum
\[\mathcal{P}\left(\hat{g}\right)\to \mathcal{P}\left(g\right) \ \ \textrm{as} \ \left|a\right|\to\infty.
\]
Now, since the reference background triple for our gluing is $(\R^{n}, \delta, 0)$ it follows at once that in our construction the components of the energy-momentum 4-vector add up exactly
\[\mathcal{E}=\sum_{i=1}^{N}\hat{\mathcal{E}}^{(i)}, \ \mathcal{P}=\sum_{i=1}^{N}\hat{\mathcal{P}}^{(i)} \]
where we are using the notation $\hat{\mathcal{E}}^{(i)}=\mathcal{E}\left(\hat{g}_{i}\right)$ and obviously $\hat{g}_{i}$ is gotten from $g_{i}$ by performing our gluing with respect to the cones of vertex $a_{i}$ and angles $\theta_{i}$ and $\theta_{i}/\left(1-\varepsilon\right)$.
As a result, in order to conclude it is enough to choose $\Lambda$ possibly a bit larger than before, namely large enough so that $\left|a_{i}\right|>\Lambda$ implies  $\left|\mathcal{E}(g_{i})-\hat{\mathcal{E}}^{(i)}\right|\leq\sigma/N$ (for every index $i=1,\ldots,N$) and analogously for the linear momentum.
\end{proof}


\

If finitely many compact subdomains $U_{1},\ldots, U_{N}$ are assigned in each $M_{1},\ldots, M_{N}$ we can exploit the gluing results by Corvino \cite{Cor00} and Corvino-Schoen \cite{CS06} to reduce to the case when every one of these is contained in a larger domain $V_{i}$ and $\partial V_{i}$ has a neighbourhood where $(g_{i}, k_{i})$ are exactly like in an annulus of a Schwarzschild or (more generally) Kerr solution. As a result, the prescribed position of $V_{1},\ldots, V_{N}$ can be described in terms of centers $x_{1},\ldots, x_{N}$ of finitely many Euclidean balls $B_{1},\ldots, B_{N}$. For each of our data $M_{1},\ldots, M_{N}$ we just need to make sure to pick an angle $\theta_{i}$ and a scaling factor $\tau$ small enough so that $\varphi(x_i,x_j)>\theta_i+\theta_j $ and of course $C_{\theta_{i}}\left(-x_{i}/\tau\right)$ contains the domain $V_{i}$. As a consequence, Theorem \ref{thm:Nbody} immediately implies constructibily of $N$-body initial data sets in the sense explained above, in substantial analogy with the main results contained in \cite{CCI09} (for the time-symmetric case) and \cite{CCI11} (for the general case). The only relevant difference with the present treatment is the behaviour at infinity of the tensors $g, k$, which in such works is prescribed to be exactly as in the Kerr solution outside a suitably large compact set, while we engineer data that are non-interacting on large time scales.

\bibliographystyle{plain}

\begin{thebibliography}{HKW}

\setcounter{footnote}{0}

\bibitem[Amb15]{Amb15} \textsc{L. C. Ambrozio}, \textit{Rigidity of area-minimizing free boundary surfaces in mean convex three-manifolds}, J. Geom. Anal. \textbf{25} (2015), no. 2, 1001-1017.

\bibitem[Bar86]{Bar86} \textsc{R. Bartnik}, \textit{The mass of an asymptotically flat manifold}, Comm. Pure Appl. Math. \textbf{39} (1986), no. 5, 661-693.

\bibitem[BBN10]{BBN10} \textsc{H. Bray S. Brendle, A. Neves}, \textit{Rigidity of area-minimizing two-spheres in three-manifolds}, Comm. Anal. Geom. \textbf{18} (2010), no. 4, 821-830. 

\bibitem[Car13]{Car13}\textsc{A. Carlotto}, \textit{Rigidity of stable minimal hypersurfaces in asymptotically flat spaces}, preprint (arXiv: 1403.6459).

\bibitem[Car14]{Car14}\textsc{A. Carlotto}, \textit{Rigidity of stable marginally outer trapped surfaces in initial data sets}, preprint (arXiv: 1404.0358).


\bibitem[CCI09]{CCI09}\textsc{P. T. Chru\'sciel, J. Corvino, J. Isenberg}, \textit{Construction of N-body time-symmetric initial data sets in general relativity}, Complex analysis and dynamical systems IV. Part 2, 83-92, Contemp. Math., 554, Amer. Math. Soc., Providence, RI, 2011. 

\bibitem[CCI11]{CCI11}\textsc{P. T. Chru\'sciel, J. Corvino, J. Isenberg}, \textit{Construction of N -body initial data sets in general relativity}, Comm. Math. Phys.  304  (2011),  no. 3, 637-647.

\bibitem[CD03]{CD03}\textsc{P. T. Chru\'sciel, E. Delay}, \textit{On mapping properties of the general relativistic constraints operator in weighted function spaces, with applications}, Mém. Soc. Math. Fr. (N.S.) No. 94 (2003), vi+103 pp.

\bibitem[CIP05]{CIP05}\textsc{P. T. Chru\'sciel, J. Isenberg, D. Pollack}, \textit{Initial data engineering}, Comm. Math. Phys. 257 (2005), no. 1, 29-42.

\bibitem[Cor00]{Cor00}\textsc{J. Corvino}, \textit{Scalar curvature deformation and a gluing construction for the Einstein constraint equations}, Comm. Math. Phys.  214  (2000),  no. 1, 137-189. 

\bibitem[CS06]{CS06}\textsc{J. Corvino, R. Schoen}, \textit{On the asymptotics for the vacuum Einstein constraint equations}, J. Differential Geom. 73 (2006), no. 2, 185-217.

\bibitem[DN55]{DN55}\textsc{A. Douglis, L. Nirenberg}, \textit{Interior estimates for elliptic systems of partial differential equations}, Comm. Pure Appl. Math. \textbf{8} (1955), 503-538.

\bibitem[EHLS11]{EHLS11}\textsc{M. Eichmair, L. H. Huang, D. A. Lee, R. Schoen}, \textit{The spacetime positive mass theorem in dimensions less than eight}, J. Eur. Math. Soc. (\textsl{to appear}).

\bibitem[FST09]{FST09}\textsc{X.-Q. Fan, Y. Shi, L.-F. Tam}, \textit{Large-sphere and small-sphere limits of the Brown-York mass}, Comm. Anal. Geom. \textbf{17} (2009), no. 1, 37-72.

\bibitem[FM73]{FM73}\textsc{A. E. Fischer, J. E. Marsden}, \textit{Linearization stability of the Einstein equations}, Bull. AMS. \textbf{79} (1973), 997-1003. 

\bibitem[FM75]{FM75}\textsc{A. E. Fischer, J. E. Marsden}, \textit{Deformations of the scalar curvature}, Duke Math. J. \textbf{42} (1975), no. 3, 519-547. 


\bibitem[MM15]{MM15}\textsc{M. Micallef, V. Moraru}, \textit{Splitting of 3-manifolds and rigidity of area-minimizing surfaces}, Proc. Amer. Math. Soc. \textbf{143} (2015), no. 7, 2865-2872. 

\bibitem[Nun13]{Nun13}\textsc{I. Nunes}, \textit{Rigidity of area-minimizing hyperbolic surfaces in three-manifolds}, J. Geom. Anal. \textbf{23} (2013), no. 3, 1290-1302. 

\bibitem[SY79]{SY79}\textsc{R. Schoen, S. T. Yau}, \textit{On the proof of the positive mass conjecture in general relativity}, Comm. Math. Phys. \textbf{65} (1979), no. 1, 45-76.


\bibitem[Wit81]{Wit81}\textsc{E. Witten}, \textit{A new proof of the positive energy theorem}, Comm. Math. Phys. \textbf{80} (1981), no. 3, 381-402. 

\end{thebibliography}

\end{document}